\documentclass{amsart}
\usepackage{amsmath}
\usepackage{amscd}
\usepackage{amssymb}
\usepackage{amsthm}
\RequirePackage{filecontents}
\usepackage{fullpage}
\usepackage{longtable}
\usepackage[numbers]{natbib}  
\usepackage[draft]{hyperref}
\theoremstyle{plain}
\newtheorem{prop}{Proposition}[section]
\newtheorem{thm}[prop]{Theorem}
\newtheorem{lem}[prop]{Lemma}
\newtheorem{cor}[prop]{Corollary}
\theoremstyle{definition}
\newtheorem{defn}[prop]{Definition}
\newtheorem{ex}[prop]{Example}
\newtheorem{rem}[prop]{Remark}

\newenvironment{psmallmatrix}
  {\left(\begin{smallmatrix}}
  {\end{smallmatrix}\right)}
  
  \newtheoremstyle{TheoremNum}
        {8.0pt plus 2.0pt minus 4.0pt}
        {8.0pt plus 2.0pt minus 4.0pt}              
        {\itshape}                      
        {}                              
        {\bfseries}                     
        {.}                             
        { .5em }                             
        {\thmname{#1}\thmnote{ \bfseries #3}}
    \theoremstyle{TheoremNum}
    \newtheorem{thm_n}{Theorem}

\author{Brandon Williams }

\address{Fachbereich Mathematik \\ Technische Universit\"at Darmstadt \\ 64289 Darmstadt, Germany}

\email{bwilliams@mathematik.tu-darmstadt.de}

\thanks{This research is supported by a postdoctoral fellowship of the LOEWE research unit Uniformized Structures in Arithmetic and Geometry.}

\begin{document}

\nocite{*}

\title{Higher pullbacks of modular forms on orthogonal groups}

\begin{abstract} We apply differential operators to modular forms on orthogonal groups $\mathrm{O}(2,\ell)$ to construct infinite families of modular forms on special cycles. These operators generalize the quasi-pullback. The subspaces of theta lifts are preserved; in particular, the higher pullbacks of the lift of a (lattice-index) Jacobi form $\phi$ are theta lifts of partial development coefficients of $\phi$.  For certain lattices of signature $(2,2)$ and $(2,3)$, for which there are interpretations as Hilbert-Siegel modular forms, we observe that the higher pullbacks coincide with differential operators introduced by Cohen and Ibukiyama.
\end{abstract}

\maketitle

\section{Introduction}

In this note we apply differential operators to modular forms on Grassmannians of signature $(2,\ell)$ lattices to construct modular forms on special cycles of arbitrary codimension. The simplest example of such an operation is the \emph{pullback}: if $\phi : L \rightarrow \Lambda$ is an isometric embedding of lattices (which induces a map on modular varieties) then one obtains a map $\phi^*$ on modular forms by setting $\phi^*(F) := F \circ \phi$. Despite the simple definition, pullbacks are a surprisingly useful tool for constructing modular forms.

The \emph{quasipullback} is a renormalization of the pullback of $F$ when $\phi^* F$ is identically zero. When $L$ (or rather its image under $\phi$) has codimension $1$ in $\Lambda$ the quasipullback of $F$ is essentially the leading term in a Taylor expansion of $F$ about the hypersurface defined by $\phi$. The quasipullbacks, especially of the weight $12$ Borcherds form $\Phi_{12}$ on the Grassmannian of $\mathrm{II}_{2,26}$, have seen important geometric applications, some of which are discussed in Chapter 8 of \cite{GHS}.

The motivation of this note is to develop a framework of linear ``higher pullback" maps $P_N$, $N \in \mathbb{N}_0$ to special cycles in which the true pullback is always $P_0$ and the quasipullback of a modular form $F$ to a Heegner divisor on which $F$ vanishes to order $N$ is (up to scalar multiple) the $N^{\text{th}}$ pullback $P_N F$. We also define higher pullbacks to special cycles of arbitrary codimension; in this case, $P_N F$ should be understood as a multilinear form which, to any $N$ normal vectors $v_1,...,v_N \in L^{\perp} \otimes \mathbb{C}$, constructs an orthogonal modular form $P_N^L F(z;v_1,...,v_n)$ on the modular variety attached to $L$ using the Laplacian on $L$ and the normal derivatives along the $v_i$:

\begin{thm_n}[\ref{pullbacks}] Let $F$ be a modular form of weight $k$ on the orthogonal Shimura variety attached to an even lattice $\Lambda$, and let $L \subseteq \Lambda$ be an even sublattice. For every order $N \in \mathbb{N}_0$, there are explicit differential operators $P_N^L$ depending linearly on $v_1,...,v_N \in L^{\perp} \otimes \mathbb{C}$ such that $P_N^L F(z;v_1,...,v_N)$ is a modular form of weight $k+N$ on the Shimura variety attached to $L$ and a cusp form if $N \ge 1$. 
\end{thm_n}

Moreover, if $F$ is the additive theta lift of a vector-valued modular form $f(\tau)$, then every value of $P_N^L F$ is also an additive theta lift. This makes it possible to compute the values $P_N^L F$ quickly in many cases of interest.

There are many instances in the literature where differential operators have been applied to construct modular forms. This note is heavily inspired by the \emph{development coefficients} of Eichler-Zagier (\cite{EZ}, Chapter 3) which produce elliptic modular forms from the Taylor coefficients of a Jacobi form. Generalizing this in some sense, Ibukiyama and Ibukiyama-Zagier (\cite{I},\cite{I2},\cite{IZ}) have applied differential operators involving \emph{higher spherical polynomials}, which generalize the classical Gegenbauer polynomials, to Hilbert-Siegel modular forms. While the Gegenbauer polynomials also appear naturally in the setting of this note, it is not clear to the author what role, if any, is played by the higher spherical polynomials in the context of orthogonal modular forms. In another direction, there are several generalizations of the \emph{Rankin-Cohen brackets} to multilinear operators on modular forms, including Rankin-Cohen brackets of (lattice-index) Jacobi forms (\cite{CE},\cite{CK}), on Siegel modular forms (\cite{EI},\cite{I}) and on orthogonal modular forms \cite{CK}.

This note is organized as follows. In section 2 we review modular forms on orthogonal groups and describe their behavior under the holomorphic Laplace operator. In section 3 we define pullbacks to Heegner divisors $\lambda^{\perp}$ as Gegenbauer polynomials in the Laplace operator on $\lambda^{\perp}$ and the directional derivative with respect to $\lambda$, and in section 4 we generalize this to arbitrary codimension. In section 5 we define partial development coefficients of lattice-index Jacobi forms and vector-valued modular forms and show that these fit into commutative diagrams involving the theta lift. In section 6 we work out examples involving quadratic forms of signature $(2,2)$ and $(2,3)$. \\

\textbf{Acknowledgments.} I thank Martin Raum for helpful comments, in particular his suggestion to consider Example 6.5.

\section{Orthogonal modular forms}

\subsection{The upper half-space} Let $(\Lambda_1,Q)$ be an even lattice of signature $(2,\ell)$ for some $\ell \in \mathbb{N}$, with induced bilinear form $$\langle x,y \rangle = Q(x+y) - Q(x) - Q(y).$$ The associated Hermitian symmetric domain is $$D^{\pm} = \Big\{ \text{norm-zero lines} \; z \in \mathbb{P}(\Lambda_1 \otimes \mathbb{C}) \; \text{such that} \; \langle z, \overline{z} \rangle > 0 \Big\}.$$ It splits into two connected components that we label $D^+$ and $D^-$.

We will almost always consider a special case which is better adapted to working with Fourier expansions: we assume there is an even lattice $(\Lambda,Q)$ of signature $(1,\ell-1)$ such that $(\Lambda_1,Q) = (\Lambda,Q) \oplus \mathrm{II}_{1,1}$ where $\mathrm{II}_{1,1}$ is the even unimodular lattice of signature $(1,1)$. (Certainly one does not obtain all interesting orthogonal modular varieties in this way; for example, certain compact Shimura curves arise from signature $(2,1)$ lattices without isotropic vectors; but most interesting cases can at least be embedded in a lattice of this type.) Without loss of generality one can take $\Lambda = \mathbb{Z}^{\ell}$ with quadratic form $Q(v) = \frac{1}{2}v^T \mathbf{S}v$ for a symmetric integral matrix $\mathbf{S}$ with even diagonal and identify $\Lambda_1 = \mathbb{Z}^{\ell + 2}$ with the Gram matrix $\mathbf{S}_1 = \begin{psmallmatrix} 0 & 0 & 1 \\ 0 & \mathbf{S} & 0 \\ 1 & 0 & 0 \end{psmallmatrix}$, which is written in blocks of size $1,\ell,1$.

Since $\Lambda$ has signature $(1,\ell-1)$ the set of positive-norm vectors in $\Lambda \otimes \mathbb{R}$ splits into two connected components, corresponding to the components of $D^{\pm}$. Suppose one component $P$ (the \emph{positive cone}) has been fixed. Then the orthogonal upper half-space attached to $\Lambda$ is $$\mathbb{H}_{\Lambda} = \{z = x+iy \in \Lambda \otimes \mathbb{C}: \; y \in P\},$$ and it embeds as an open dense subset of a component of $D^{\pm}$ (which we label $D^+$) by $$z \mapsto \mathrm{span}(-Q(z),z,1).$$ The group $G = \mathrm{SO}^+(\Lambda_1 \otimes \mathbb{R})$, i.e. the connected component of the identity in the orthogonal group of $\Lambda_1 \otimes \mathbb{R}$, acts on $\mathbb{H}_{\Lambda}$ by M\"obius transformations as follows. For any $M \in G$, understood as a matrix preserving $\mathbf{S}_1$, and $z \in \mathbb{H}_{\Lambda}$ there is a unique scalar $j(M;z) \in \mathbb{C}^{\times}$ and vector $w = M \cdot z \in \mathbb{H}_{\Lambda}$ such that $$M \begin{psmallmatrix} -Q(z) \\ z \\ 1 \end{psmallmatrix} = j(M;z) \begin{psmallmatrix} -Q(w) \\ w \\ 1 \end{psmallmatrix}.$$

The \emph{modular group} $\Gamma_{\Lambda}$ is the intersection of $G$ with the discriminant kernel of $\Lambda_1$. In other words, $$\Gamma_{\Lambda} = \{M \in G: \; M \; \text{acts trivially on} \; \Lambda_1' / \Lambda_1\}.$$ Elements $M \in \Gamma_{\Lambda}$ may be understood as those matrices $M \in \mathbb{Z}^{(\ell+2) \times (\ell+2)}$ such that $M^T \mathbf{S}_1 M = \mathbf{S}_1$ and $\mathrm{det}(M) = 1$ and $(M - I) \mathbf{S}_1^{-1}$ has integer entries.

It is difficult to produce generators of $\Gamma_{\Lambda}.$ However the differential operators we are concerned with behave nicely under the entire Lie group $G$, for which one can give a system of generators with simple factors of automorphy. The following result is well-known but we include it here for convenience.

\begin{prop} The group $G = \mathrm{SO}^+(\Lambda_1 \otimes \mathbb{R})$ is generated by operators of the following types. \\
(i) Translations: for any $b \in \Lambda \otimes \mathbb{R}$, define $T_b = \begin{psmallmatrix} 1 & -b^T \mathbf{S} & -Q(b) \\ 0 & I & b \\ 0 & 0 & 1 \end{psmallmatrix}$ which maps $z$ to $z+b$ and has cocycle of automorphy $j(T_b;z) = 1$. \\
(ii) Rotations: for any $A \in \mathrm{SO}^+(\Lambda)$ which acts trivially on $\Lambda'/\Lambda$, define $R_A = \begin{psmallmatrix} 1 & 0 & 0 \\ 0 & A & 0 \\ 0 & 0 & 1 \end{psmallmatrix}$ which maps $z$ to $Az$ and has cocycle of automorphy $j(R_A;z) = 1$. \\
(iii) Scalings: for any $t \in \mathbb{R}_{>0}$, define $S_t = \begin{psmallmatrix} t & 0 & 0 \\ 0 & I & 0 \\ 0 & 0 & t^{-1} \end{psmallmatrix}$ which maps $z$ to $tz$ with cocycle of automorphy $j(S_t;z) = t^{-1}$. \\
(iv) Inversions: for any $v \in \Lambda \otimes \mathbb{R}$ of norm $Q(v) = 1$, define $J_v = \begin{psmallmatrix} 0 & 0 & -1 \\ 0 & I - vv^T \mathbf{S} & 0 \\ -1 & 0 & 0 \end{psmallmatrix}$ which maps $z$ to $\frac{z - \langle v,z \rangle v}{Q(z)}$ with cocycle of automorphy $j(J_v;z) = Q(z)$.
\end{prop}
\begin{proof} It is straightforward to check that the matrices above are orthogonal with respect to $\mathbf{S}_1$ and act on $\mathbb{H}_{\Lambda}$ as described. $G$ is connected so we can generate it by exponentiating matrices which span its Lie algebra $\mathfrak{g}$. We decompose $$\mathfrak{g} = \begin{psmallmatrix} 0 & 0 & 0 \\ 0 & \mathfrak{so}(\Lambda) & 0 \\ 0 & 0 & 0 \end{psmallmatrix} \oplus \Big\{ \begin{psmallmatrix} 0 & -b^T \mathbf{S} & 0 \\ 0 & 0 & b \\ 0 & 0 & 0 \end{psmallmatrix}: \; b \in \Lambda \otimes \mathbb{R} \Big\}  \oplus \mathbb{R} \cdot \begin{psmallmatrix} 1 & 0 & 0 \\ 0 & 0 & 0 \\ 0 & 0 & -1 \end{psmallmatrix} \oplus \Big\{ \begin{psmallmatrix} 0 & 0 & 0 \\ b & 0 & 0 \\ 0 & -b^T \mathbf{S} & 0 \end{psmallmatrix}: \; b \in \Lambda \otimes \mathbb{R} \Big\}.$$ The first summand exponentiates to rotations; the second to translations; the third to scalings; and the fourth to matrices of the form $$T_b^* = \begin{psmallmatrix} 1 & 0 & 0 \\ b & I & 0 \\ -Q(b) & -b^T \mathbf{S} & 1 \end{psmallmatrix}$$ which one obtains through the identity $T_{-R_v b}^* = J_v T_b J_v$ where $R_v = I - vv^T \mathbf{S}$ is the reflection through any vector $v \in \Lambda \otimes\mathbb{R}$ with $Q(v) = 1$.
\end{proof}

\begin{ex} It is often helpful to consider the simplest example: a lattice $\Lambda$ generated by a single vector $v$ of norm $1$, i.e. with Gram matrix $\mathbf{S} = (2)$. In this case $\mathrm{SO}^+(\Lambda_1 \otimes \mathbb{R})$ is isomorphic to $\mathrm{PSL}_2(\mathbb{R})$ via the latter's adjoint representation on its Lie algebra. The upper half-space $\mathbb{H}_{\Lambda}$ is identified with the usual upper half-plane $\mathbb{H} = \{z = x+iy: \, y > 0\}$ in an obvious way. Through this identification the standard generators $T = \begin{psmallmatrix} 1 & 1 \\ 0 & 1 \end{psmallmatrix}$ and $S = \begin{psmallmatrix} 0 & -1 \\ 1 & 0 \end{psmallmatrix}$ of $\mathrm{PSL}_2(\mathbb{Z})$ correspond to the translation $T_v$ and the inversion $J_v$, respectively.
\end{ex}

\subsection{Modular forms} For a general signature $(2,\ell)$ lattice $\Lambda_1$, we take the tautological bundle $\pi : E \rightarrow D^+$ where $E = \{z \in \Lambda_1 \otimes \mathbb{C}: \; \mathrm{span}(z) \in D^+\}$ and define \emph{automorphic forms} to be meromorphic functions $F : E \rightarrow \mathbb{C}$ which are homogeneous (i.e. $F(tz) = t^{-k} F(z)$ for some $k \in \mathbb{Z}$, called the weight of $F$) and which are invariant under the modular group: $F(Mz) = F(z)$ for all $M$ in the discriminant kernel of $\Lambda_1$.

In the case $\Lambda_1 = \Lambda \oplus \mathrm{II}_{1,1}$ it is enough to define automorphic forms as functions on the open dense subset $\mathbb{H}_{\Lambda} \subseteq D^+$. On functions $F : \mathbb{H}_{\Lambda} \rightarrow \mathbb{C}$ we define the slash operator $$F \Big|_k M(z) = j(M;z)^{-k} F(M \cdot z), \; \; M \in G, \; k \in \mathbb{Z}.$$ Then the definition of an automorphic form of weight $k$ reduces to a meromorphic function $F$ satisfying the functional equations $F|_k M = F$ for all $M \in \Gamma_{\Lambda}$. We call $F$ a \emph{modular form} if it is holomorphic on $\mathbb{H}_{\Lambda}$ and has bounded growth at cusps in the sense that the limit $\lim_{t \rightarrow \infty} (F|_k M)(itv)$ is bounded for all $M \in \mathrm{SO}^+(\Lambda_1 \otimes \mathbb{Q})$ and all positive vectors $v \in P$. Note that the growth conditions are automatically satisfied if $\ell \ge 3$, or if $\ell = 2$ and $\Lambda$ is anisotropic (Koecher's principle). Moreover $F$ is a \emph{cusp form} if those limits are zero.

Orthogonal modular forms are invariant under translations by $\Lambda$ and therefore have Fourier expansions: $$F(z) = \sum_{\lambda \in \Lambda'} c(\lambda) \mathbf{q}^{\lambda}, \; \; \mathbf{q}^{\lambda} = e^{2\pi i \langle \lambda, z \rangle}, \; c(\lambda) \in \mathbb{C}.$$ The growth condition implies that $c(\lambda) = 0$ unless $\lambda \in \overline{P}$ (the closure of the positive cone). Cusp forms have their Fourier coefficients supported on the positive cone itself. \\

\subsection{The Laplacian} With respect to the Gram matrix $\mathbf{S}$ which was fixed above, write $\mathbf{S}^{-1} = (s^{ij})_{i,j=1}^n$ and define the \emph{Laplace operator} on $\mathbb{H}_{\Lambda}$ by $$\Delta = \frac{1}{2} \sum_{i,j} s^{ij} \frac{\partial^2}{\partial z_i \partial z_j}.$$ (This is an abuse of nomenclature as $\mathbf{S}$ is not positive-definite.) Additionally define the \emph{Euler operator} by $$\mathcal{E} = z^T \nabla = \sum_i z_i \frac{\partial}{\partial z_i}.$$ In the lemma below we collect some basic properties.
\begin{lem} (i) $\Delta \mathbf{q}^{\lambda} = (2\pi i)^2 Q(\lambda) \mathbf{q}^{\lambda}$ for any $\lambda \in \Lambda'$. \\
(ii) $\Delta Q(z)^{-k} = k (1 + k - \ell/2) Q(z)^{-k-1}$ for any $k \in \mathbb{R}$. \\
(iii) $\mathcal{E} \Delta = \Delta (\mathcal{E} - 2)$.
\end{lem}
\begin{proof} (i) This is because $\Delta \mathbf{q}^{\lambda} = (1/2) \nabla \cdot 2\pi i e^{2\pi i \lambda^T \mathbf{S}z } = (2\pi i)^2 Q(\lambda) e^{2\pi i \lambda^T \mathbf{S}z}.$ \\
(ii) This is because $$\Delta Q(z)^{-k} = \frac{1}{2} \nabla \cdot (-k Q(z)^{-k-1} z) = \frac{-k \ell}{2} Q(z)^{-k-1} + \frac{k(k+1)}{2} Q(z)^{-k-2} z^T \mathbf{S}z = k (1 + k - \ell/2) Q(z)^{-k-1}.$$
(iii) This is because \[ \Delta \mathcal{E} = \frac{1}{2} \sum_{i,j,k} s^{ij} \frac{\partial^2}{\partial z_i \partial z_j} \Big( z_k \frac{\partial}{\partial z_k} \Big)= \frac{1}{2} \sum_{i,j,k} s^{ij} \Big( \delta_{jk} \frac{\partial^2}{\partial z_i \partial z_k} + \delta_{ik} \frac{\partial^2}{\partial z_j \partial z_k} + z_k \frac{\partial^3}{\partial z_i \partial z_j \partial z_k} \Big) = 2 \Delta + \mathcal{E} \Delta. \qedhere \]
\end{proof}
Part (iii) is generalized by the fact that the map $X = \begin{psmallmatrix} 0 & 1 \\ 0 & 0 \end{psmallmatrix} \mapsto [\text{multiplication by} \, -Q(z)]$, $Y = \begin{psmallmatrix} 0 & 0 \\ 1 & 0 \end{psmallmatrix} \mapsto \Delta$ and $H = \begin{psmallmatrix} 1 & 0 \\ 0 & -1 \end{psmallmatrix} \mapsto \mathcal{E} + \ell/2$ determines a representation of $\mathfrak{sl}_2(\mathbb{Z})$, but we will not use this.

The following lemma says that applying $\Delta$ to a modular form $n$ times ``almost" raises its weight by $2n$.

\begin{lem} Let $F : \mathbb{H}_{\Lambda} \rightarrow \mathbb{C}$ be a holomorphic function and $k,n \in \mathbb{N}$. Then: \\
(i) $\Delta^n (F|_k T_b) = (\Delta^n F)|_{k+2n} T_b$; \\
(ii) $\Delta^n (F|_k R_A) = (\Delta^n F)|_{k+2n} R_A$; \\
(iii) $\Delta^n (F |_k S_t) = (\Delta^n F)|_{k+2n} S_t$; \\
(iv) $$\Delta^n (F|_k J_v) = \sum_{j=0}^n \frac{\Gamma(n+k+1-\ell/2)}{\Gamma(j+k+1-\ell/2)} \binom{n}{j} \Big( \Delta^j (\mathcal{E}+k)^{(n-j)} F \Big) \Big|_{k+n+j} J_v,$$ where $(\mathcal{E}+k)^{(n-j)}$ denotes the rising factorial $$(\mathcal{E} + k)^{(n-j)} = (\mathcal{E} + k)(\mathcal{E}+k+1)...(\mathcal{E}+k+n-j-1)$$ as a composition of operators.
\end{lem}
Here $\frac{\Gamma(n+k+1-\ell/2)}{\Gamma(j+k+1-\ell/2)} = (j+k+1-\ell/2)...(n+k-\ell/2)$ is well-defined even if $n+k+1-\ell/2 \in -\mathbb{N}_0$.

\begin{rem} (i) The terms $\frac{\Gamma(n+k+1-\ell/2)}{\Gamma(j+k+1-\ell/2)} \binom{n}{j}$ appear in many contexts involving differentiation and M\"obius transformations; compare the proof of Proposition 19 in \cite{Z}. \\
(ii) The terms $\frac{\Gamma(n+k+1-\ell/2)}{\Gamma(j+k+1-\ell/2)} \binom{n}{j} (\mathcal{E} + k) F \Big|_{k+1} J_v$ corresponding to indices $j < n$ in $\Delta (F|_k J_v)$ may be thought of as obstructions to modularity of $\Delta^n F$, if $F$ was a moduar form of weight $k$. These obstructions vanish precisely in weight $k = \ell/2 - n$. In this weight all modular forms $F$ are \emph{singular}, that is, their Fourier expansions are supported on norm-zero vectors, or equivalently $\Delta^n F = 0$. If one takes $F$ to be a meromorphic modular form of this weight then the modularity of $\Delta^n F$ is a form of Bol's identity. 
\end{rem}

\begin{proof} Parts (i) through (iii) are easy because $\Delta$ preserves translations, scalings and the orthogonal group of $\mathbf{S}$. To prove (iv) we simply use the chain rule and induction on $n$. When $n=1$ this reduces to proving $$\Delta (F|_k J_v) = (1 + k - \ell/2) \Big( (\mathcal{E} + k) F \Big) \Big|_{k+1} J_v + (\Delta F) \Big|_{k+2} J_v.$$ We first take $z$ to have imaginary part in the negative cone (so $F(z/Q(z))$ is well-defined) and compute $$\nabla \Big[ z \mapsto F(z/Q(z)) \Big] = Q(z)^{-1} (I - Q(z)^{-1} \mathbf{S}zz^T) (\nabla F)(z / Q(z))$$ and \begin{align*} \Delta \Big[ z \mapsto F(z/Q(z)) \Big] &= \frac{1}{2} \nabla \cdot Q(z)^{-1} (\mathbf{S}^{-1} - Q(z)^{-1} zz^T) (\nabla F)(z / Q(z)) \\ &= (1 - \ell/2) Q(z)^{-2} z^T (\nabla F)(z/Q(z)) + Q(z)^{-2} (\Delta F)(z / Q(z)). \end{align*} Using the product rule for $\Delta$ we obtain \begin{align*} &\quad \Delta \Big[ z \mapsto Q(z)^{-1} F(z / Q(z)) \Big] \\ &= \Delta \Big[Q(z)^{-k}\Big] F(z / Q(z)) + Q(z)^{-k} \Delta \Big[ z \mapsto F(z/Q(z)) \Big] \\ &\quad \quad + \langle \mathbf{S}^{-1} \nabla Q(z)^{-k}, \mathbf{S}^{-1} \nabla[z \mapsto F(z/Q(z))] \rangle \\ &= k (1+k-\ell/2) Q(z)^{-k-1} + (1 - \ell/2) Q(z)^{-k-2} z^T (\nabla F)(z/Q(z)) \\ &\quad \quad + Q(z)^{-k-2} (\Delta F)(z / Q(z)) - k Q(z)^{-k-2} z^T (I - \mathbf{S}zz^T / Q(z)) (\nabla F)(z / Q(z)) \\ &= k (1+k-\ell/2) Q(z)^{-k-1} + (1+ k - \ell/2) (\mathcal{E} F)(z / Q(z)) + Q(z)^{-k-2} (\Delta F)(z / Q(z)). \end{align*} Now replace $z$ by its reflection $z - \langle v,z \rangle v$ (which is an orthogonal reflection with respect to $\mathbf{S}$ and therefore leaves $\Delta$ invariant) to obtain, for $z \in \mathbb{H}_{\Lambda}$, $$\Delta(F |_k J_v) = k (1 + k - \ell/2) F|_{k+1} J_v + (1 + k - \ell/2) (\mathcal{E} F)|_{k+1} J_v + (\Delta F) |_{k+2} J_v.$$

In general, suppose we have found an identity of the form $$\Delta^n (F|_k J_v) = \sum_{j=0}^n c(n,j,k) \Big( \Delta^j (\mathcal{E} + k)^{(n-j)} F \Big) \Big|_{k+n+j} J_v$$ for some constants $c(n,j,k)$. Using the paragraph above and the relation $\mathcal{E} \Delta = \Delta (\mathcal{E} - 2)$ we find \begin{align*} \Delta^{n+1} (F|_k J_v) &= \sum_{j=0}^n c(n,j,k) (1 + k + n + j - \ell/2) \Big( (\mathcal{E} + k + n + j) \Delta^j (\mathcal{E} + k)^{(n-j)} F \Big) \Big|_{k+n+1+j} J_v \\ &\quad \quad + \sum_{j=0}^n c(n,j,k) \Big( \Delta^{j+1} (\mathcal{E} + k)^{(n-j)} F \Big) \Big|_{k+n+2+j} J_v \\ &= \sum_{j=0}^{n+1} \Big( (1 + k + n + j - \ell/2) c(n,j,k) + c(n,j-1,k) \Big) \Big( \Delta^j (\mathcal{E} + k)^{(n+1-j)} F \Big) \Big|_{k+n+1+j} J_v, \end{align*} so we find an identity of this form for $\Delta^{n+1}(F|_k J_v)$ with coefficients determined by the recurrence $$c(n+1,j,k) = (1 + k + n + j - \ell/2) c(n,j,k) + c(n,j-1,k), \; \; c(0,j,k) = \begin{cases} 1: & j = 0; \\ 0: & \text{otherwise}. \end{cases}$$ This recurrence is solved by $\frac{\Gamma(n+1+k-\ell/2)}{\Gamma(j+1+k-\ell/2)} \binom{n}{j}$ as one can verify (or reduce to the case $1+k-\ell/2=0$ and observe that the recurrence defines the unsigned Lah numbers which have closed form $\frac{n!}{j!} \binom{n-1}{j-1}$), so we obtain the claim.
\end{proof}

\begin{ex} In the case of elliptic modular forms of level $1$, we take $\ell = 1$ and $\Delta = \frac{1}{4}\frac{d^2}{d \tau^2}$ and $\mathcal{E} = \tau \frac{d}{d \tau}$ and this formula takes the form $$\frac{d^{2n}}{d\tau^{2n}} \Big( \tau^{-2k} f(-1/\tau) \Big) = \sum_{j=0}^n  \frac{2^{n-j} \tau^{-2k-2n-2j} (2n+2k-1)!! n!}{(2j+2k-1)!! j! (n-j)!} \Big( \frac{d^{2j}}{d\tau^{2j}} \Big( \tau \frac{d}{d\tau} + k \Big) ... \Big( \tau \frac{d}{d\tau} + k+n-j-1 \Big) f \Big)(-1/\tau).$$ Here $(2n-1)!! = (2n-1)(2n-3)...1$ and $(-1)!! = 1$.
\end{ex}

\section{Pullbacks to Heegner divisors}

Let $(\Lambda,Q)$ be an even lattice of signature $(1,\ell-1)$ and let $\lambda \in \Lambda$ be a lattice vector of negative norm $Q(\lambda) = -m$. Let $L$ be the orthogonal complement $\lambda^{\perp}$ in $\Lambda$. Then $L \oplus \mathbb{Z}\lambda \subseteq \Lambda$ is a sublattice of full rank and we write $\mathbb{H}_{\Lambda} = \mathbb{H}_{L \oplus \mathbb{Z}\lambda}$ in coordinates $$\mathbb{H}_{\Lambda} = \{z + w\lambda : \; z = x+iy \in \mathbb{H}_L, \, w = u+iv \in \mathbb{C}, \, v^2 < Q_L(y) / m\}.$$

To any $M \in \mathrm{SO}^+((L \oplus \mathrm{II}_{1,1}) \otimes \mathbb{R})$ let $\tilde M \in \mathrm{SO}^+((\Lambda \oplus \mathrm{II}_{1,1}) \otimes \mathbb{R})$ be the orthogonal matrix which restricts to $M$ on $L \oplus \mathrm{II}_{1,1}$ and which leaves $\lambda$ fixed. If $M$ comes from the modular group $\Gamma_L$ then $\tilde M$ acts trivially on the discriminant group $(L \oplus \mathbb{Z} \lambda)'/(L \oplus \mathbb{Z}\lambda)$. Since $\Lambda, \Lambda' \subseteq (L \oplus \mathbb{Z}\lambda)'$ it follows that $\tilde M$ maps $\Lambda$ into itself and acts trivially on $\Lambda'/\Lambda$, i.e. $\tilde M \in \Gamma_{\Lambda}$. In this way the embedding $$\mathbb{H}_L \longrightarrow \mathbb{H}_{\Lambda}, \; \; z \mapsto z + 0 \cdot \lambda$$ identifies $X_L = \overline{\Gamma_L \backslash \mathbb{H}_L}$ with an analytic divisor on $X_{\Lambda} = \overline{\Gamma_{\Lambda} \backslash \mathbb{H}_{\Lambda}}$. Linear combinations of divisors that arise in this way are called \emph{Heegner divisors}.

\begin{ex} Let $\Lambda = \mathrm{II}_{1,1}$ be a hyperbolic plane. One can identify $\Lambda \oplus \mathrm{II}_{1,1}$ with the lattice of integral $(2 \times 2)$ matrices with quadratic form given by the determinant. The group $\Gamma_{\Lambda}$ is generated by the transpose and by $\mathrm{SL}_2(\mathbb{Z}) \times \mathrm{SL}_2(\mathbb{Z}) / \{\pm (I,I)\}$ where $(M,N) \in \mathrm{SL}_2(\mathbb{Z}) \times \mathrm{SL}_2(\mathbb{Z})$ acts by left-multiplication by $M^T$ and by right-multiplication by $N$, and $X_{\Lambda}$ is the product of modular curves $X(1) \times X(1)$ modulo $(\tau_1,\tau_2) \sim (\tau_2,\tau_1)$. If $p$ is a prime and $\lambda = (-p,1)$ then $L = \mathbb{Z}(1,p)$ is a one-dimensional lattice generated by a vector of norm $p$; and $X_L$ is the curve $X_0(p)$ modulo the Fricke involution $\tau \mapsto -1/p\tau$; and the embedding $X_L \rightarrow X_{\Lambda}$ above is essentially the $p^{\text{th}}$ Hecke correspondence.
\end{ex}

Interpret $\mathbb{H}_{\Lambda}$ as a subset of $\mathbb{H}_L \times \mathbb{C}$ with $z+w \lambda$ corresponding to the pair $(z,w)$. We will define pullback operators from modular forms on $\mathbb{H}_{\Lambda}$ to modular forms on $\mathbb{H}_L$ by evaluating certain Gegenbauer polynomials in the partial derivative $\partial_w$ and the Laplace operator $\Delta_L$ on $\mathbb{H}_L$ along the divisor $w=0$. \\

The main point is to understand how the partial derivative $\partial_w$ behaves with respect to inversions of $\mathbb{H}_L$. Let $\mathcal{E}_z$ denote the Euler operator in the variable $z$, i.e. $\mathcal{E}_z = \sum_{i=1}^{\ell - 1} z_i \partial_{z_i}$ and let $(\mathcal{E}_z + k)^{(n)}$ denote the rising factorials as before.

\begin{lem} Let $F : \mathbb{H}_{\Lambda} \rightarrow \mathbb{C}$ be a holomorphic function and let $v \in L \otimes \mathbb{R}$ with $Q(v) = 1$. For any $n \in \mathbb{N}_0$, $$\partial_w^n \Big|_{w=0}  \Big( F \Big|_k J_v \Big) = \sum_{i=0}^{\lfloor n/2 \rfloor} \frac{n!}{i! (n-2i)!} m^i \Big( (\mathcal{E}_z + k + n - 2i)^{(i)} \partial_w^{n-2i} \Big|_{w=0} F \Big) \Big|_{k+n-i} J_v.$$
\end{lem}

For example, when $n=0,1,2$ this reduces to $$(F |_k J_v) \Big|_{w=0} = \Big( F\Big|_{w=0} \Big) \Big|_k J_v; \; \; \partial_w \Big|_{w=0} \Big( F \Big|_k J_v \Big) = \Big( \partial_w \Big|_{w=0} F \Big) \Big|_{k+1} J_v;$$ $$\partial_w^2 \Big|_{w=0} \Big( F \Big|_k J_v \Big) = 2m \Big( (\mathcal{E}_z + k ) F \Big) \Big|_{k+1} J_v + \Big( \partial_w^2 \Big|_{w=0} F \Big) \Big|_{k+2} J_v.$$ In fact, evaluating a modular form for $\Lambda$ of weight $k$ at $w=0$ yields a modular form for $L$ of weight $k$; and evaluating its partial derivative $\partial_w F$ at $w=0$ yields a modular form for $L$ of weight $k+1$. \\

In the proof we use some standard notation for multivariate power series. If $z = (z_1,...,z_n) \in \mathbb{C}^n$ and $\alpha = (\alpha_1,...,\alpha_n) \in (\mathbb{N}_0)^n$ then define $$\alpha! = \alpha_1! \cdot ... \cdot \alpha_n!; \; |\alpha| = \alpha_1+...+\alpha_n; \; z^{\alpha} = z_1^{\alpha_1} \cdot ... \cdot z_n^{\alpha_n}.$$
\begin{proof} Both sides of the claim are formally invariant under translation by arbitrary vectors in $\Lambda \otimes \mathbb{C}$ so it is enough to prove this for functions which are analytic in a neighborhood of $(z,w) = (0,0)$ in $\Lambda \otimes \mathbb{C}$. Expand $F$ as a Taylor series: $$F(z,w) = \sum_{\alpha \in \mathbb{N}^{\ell - 1}} \sum_{j=0}^{\infty} \frac{c(\alpha,j)}{\alpha! j!} z^{\alpha} w^j.$$ Let $R_v z = z - \langle v,z \rangle v$ denote the reflection along $v$. Using the binomial theorem for $w$ and $\|z\|$ sufficiently small we find \begin{align*} F \Big|_k J_v &= \sum_{\alpha} \sum_{j=0}^{\infty} \frac{c(\alpha,j)}{\alpha! j!} (R_v z)^{\alpha} w^j (Q(z) - mw^2)^{-|\alpha|-j-k} \\ &= \sum_{\alpha} \sum_{i,j=0}^{\infty} \binom{|\alpha|+j+k+i-1}{i} m^i Q(z)^{-|\alpha|-j-k-i} \frac{c(\alpha,j)}{\alpha! j!} (R_v z)^{\alpha} w^{j+2i}. \end{align*} We apply $\partial_w^n$ and set $w=0$ to find $$\partial_w^n \Big|_{w=0} (F|_k J_v) = \sum_{i=0}^{\lfloor n/2 \rfloor} \frac{n!}{(n-2i)!} m^i \sum_{\alpha} \binom{|\alpha|+n+k-i-1}{i} \frac{c(\alpha,n-2i)}{\alpha!} Q(z)^{-n-k+i} (J_v \cdot z)^{\alpha}.$$ The claim follows by observing that $\mathcal{E}_z$ acts on power series by $\mathcal{E}_z z^{\alpha} = |\alpha| z^{\alpha}$, and therefore \[\frac{1}{i!} (\mathcal{E}_z + k + n - 2i)^{(i)} z^{\alpha} = \binom{|\alpha|+n+k-i-1}{i} z^{\alpha}. \qedhere \]
\end{proof}

\begin{cor} Let $F : \mathbb{H}_{\Lambda} \rightarrow \mathbb{C}$ be holomorphic and let $v \in L$ with $Q(v) = 1$. For $n_1,n_2 \in \mathbb{N}_0$,

\begin{align*} \Delta_L^{n_1} \partial_w^{n_2} \Big|_{w=0} \Big( F \Big|_k J_v \Big) &= \sum_{i=0}^{\lfloor n_2/2 \rfloor} \sum_{j=0}^{n_1} \frac{m^i n_1! n_2! \Gamma(k+n_1+n_2 + \frac{3-\ell}{2} - i)}{i!j!(n_1-j)!(n_2-2i)! \Gamma(j+k+n_2 + \frac{3-\ell}{2} - i)} \\ &\quad\quad \times \Big( \Delta_L^j (\mathcal{E}+k+n_2 - 2i)^{(n_1 + i - j)} \partial_w^{n_2 - 2i} \Big|_{w=0} F \Big) \Big|_{k+n_1+n_2+j-i} J_v. \end{align*}
\end{cor}
\begin{proof} Combine Lemma 2.4 and Lemma 3.2.
\end{proof}

\begin{defn} For $s \in \mathbb{C}$ and $N \in \mathbb{N}_0$, define the homogeneized \emph{Gegenbauer polynomial} $g_N^s(x,y)$ in the variables $x,y$ by $$g_N^s(x,y) =  \sum_{\substack{n_1,n_2 \in \mathbb{N}_0 \\ 2n_1 + n_2 = N}} (-1)^{n_1} \frac{\Gamma(s+n_1+n_2)}{\Gamma(s) n_1! n_2!} x^{n_2} y^{n_1},$$ with generating function $$\sum_{N=0}^{\infty} g_N^s(x,y) t^N = (1 - xt + yt^2)^{-s}.$$
\end{defn}

We will generally use the rescaling $$G_N^s(x,y) := \frac{N! \Gamma(s)}{\Gamma(s + \lceil N/2 \rceil)} g_N^s(x,y),$$ which has better integrality properties (and is compatible with the Gegenbauer polynomials as used by Eichler and Zagier \cite{EZ}): namely $G_N^s(x,y) \in \mathbb{Z}[x,y,s]$.

\begin{thm}  For holomorphic $F : \mathbb{H}_{\Lambda} \rightarrow \mathbb{C}$ define the \emph{$N^{th}$ pullback} to $\mathbb{H}_L$ by $$P_{N;k}^L F(z;\lambda) = (2\pi i)^{-N} G_N^{k + (1-\ell)/2}(\partial_w, m \Delta_L) \Big|_{w=0} F(z).$$ For any $v \in L \otimes \mathbb{R}$ with $Q(v) = 1$, $$P_N^L \Big( F \Big|_k J_v \Big) = \Big( P_N^L F \Big) \Big|_{k+N} J_v.$$
\end{thm}
Here we are interpreting $J_v$ as an inversion of either $\mathbb{H}_L$ or $\mathbb{H}_{\Lambda}$ as the context requires. The notation emphasizes the dependence on the vector $\lambda \in L^{\perp}$ which was chosen at the beginning of the section. We will usually suppress the weight $k$ if it is clear from the context.
\begin{proof}  Using Corollary 3.3 we compute, up to the constant multiple $C := (2\pi i)^{-N} \frac{N!}{\Gamma(k + (1 - \ell)/2 + \lceil N/2 \rceil)}$, \begin{align*} &\quad P_{N}^L (F |_k J_v)(\tau;\lambda) \\ &= C \sum_{2n_1 + n_2 = N} \sum_{j=0}^{n_1} \sum_{i=0}^{\lfloor n_2/2 \rfloor} \frac{(-1)^{n_1} m^{i+n_1} \Gamma(N - n_1 + k + \frac{1 - \ell}{2}) \Gamma(k+n_1+n_2 + \frac{3-\ell}{2} - i)}{ i! j! (n_1 - j)! (n_2 - 2i)! \Gamma(j+k+n_2 + \frac{3-\ell}{2} - i)} \\ &\quad\quad \times \Big( \Delta_L^j (\mathcal{E} + k + n_2 - 2i)^{(n_1 + i - j)} \partial_w^{n_2 - 2i} \Big|_{w=0} F \Big) \Big|_{k+n_1 + n_2 + j-i} J_v \\ &= C \sum_{u=0}^{\lfloor N/2 \rfloor} \frac{(-m)^{u} \Gamma(k+N-u + \frac{3-\ell}{2})}{(N - 2 u)!} \sum_{j=0}^{u} \frac{1}{j! (u - j)!} \sum_{i=0}^{u - j} (-1)^i \binom{u - j}{i} \frac{\Gamma(N+k + \frac{1 - \ell}{2}-u + i)}{\Gamma(N+k + \frac{3 - \ell}{2} -u + i - (u - j))} \\ &\quad \quad \times \Big(\Delta_L^j (\mathcal{E} + k + N - 2 u)^{(u - j)} \partial_w^{N - 2 u} \Big|_{w=0} F \Big) \Big|_{k + N + j - u} J_v, \end{align*} where we have labelled $u := n_1 + i$. The inner sum over $i$ simplifies as follows. For any $n > 0$ and $s \in \mathbb{C} \backslash \{0\}$, use the binomial theorem to see that $$\sum_{i=0}^{n} (-1)^i \binom{n}{i} \frac{\Gamma(i+s)}{\Gamma(i+s+1 -n)} = \frac{d^{n-1}}{dt^{n-1}} \Big|_{t=1} t^{s-1} (1 - t)^n = 0.$$ Of course when $n=0$ we obtain $\sum_{i=0}^n (-1)^i \binom{n}{i} \frac{\Gamma(i+s)}{\Gamma(i+s+1-n)} = 1/s$. Setting $n = u - j$ and $s = N+k+\frac{1-\ell}{2}-u$, we find \begin{align*} P_N^L(F|_k J_v) &= C \sum_{u = 0}^{\lfloor N/2 \rfloor} \frac{(-m)^{u} \Gamma(N+k - u + \frac{3 - \ell}{2})}{(N+k-u + \frac{1-\ell}{2}) u! (N - 2u)!} \Big( \Delta_L^{u} \partial_w^{n_2 - 2 u} \Big|_{w=0} F \Big) \Big|_{k+N} J_v \\ &= (P_N^L F)(\tau;\lambda) \Big|_{k+N} J_v. \qedhere \end{align*}
\end{proof}

\begin{cor} Let $F : \mathbb{H}_{\Lambda} \rightarrow \mathbb{C}$ be an orthogonal modular form of weight $k$. Then $P_N^L F$ is an orthogonal modular form of weight $k+N$ on $\mathbb{H}_L$ and it is a cusp form if $N \ge 1$. If $F$ has Fourier expansion $$F(z+w \lambda) = \sum_{r \in L'} \sum_{\mu \in (2m)^{-1} \mathbb{Z}} c(r,\mu) \mathbf{q}^{r} s^{\mu}, \; \; \mathbf{q}^{r} = e^{2\pi i \langle r, z \rangle}, \; s = e^{2\pi i (2m) w}$$ then $P_N F$ has Fourier expansion $$P_N^L F(z;\lambda) = \sum_{r \in L'} \Big( \sum_{\mu \in (2m)^{-1}\mathbb{Z}} c(r,\mu) G_N^{k+(1-\ell)/2}(2m \mu,mQ(r)) \Big) \mathbf{q}^{r}.$$ Moreover, $P_N^L F(z;\lambda)$ is homogeneous of degree $N$ in $\lambda$, i.e. for any $a \in \mathbb{Z}$, $$P_N^L F(z; a \lambda) = a^N \cdot P_N^L F(z;\lambda).$$
\end{cor}
\begin{proof} Both $\Delta_L$ and $\partial_w^n |_{w=0}$ are invariant under translations, rotations and scalings of $\mathbb{H}_L$, so we find $$P_N^L (F|_k \tilde M)(\tau;\lambda) = (P_N^L F(\tau;\lambda)) \Big|_{k+N} M$$ for all holomorphic $F$ and all matrices $M$ of these types. ($\tilde M$ is as defined in the beginning of this section.) Together with the previous lemma this implies that $$P_N^L (F|_k \tilde M) (\tau;\lambda) = (P_N^L F(\tau;\lambda))|_{k+N} M \; \text{for all} \; M \in \mathrm{SO}^+((L \oplus \mathrm{II}_{1,1}) \otimes \mathbb{R}).$$ If $F$ was a modular form of weight $k$ then this implies that $P_N^L F(\tau;\lambda)$ is modular and has weight $k+N$.

The assertion about Fourier expansions holds because $\Delta_L^{n_1} (\mathbf{q}^r s^{\mu})|_{w=0} = (2\pi i)^{2n_1} Q(r)^{n_1} \mathbf{q}^r$ and because $\partial_w^{n_2} (\mathbf{q}^r s^{\mu})|_{w=0} = (2\pi i)^{n_2} (2m\mu)^{n_2} \mathbf{q}^r$. To see that $P_N^L F(z;\lambda)$ is a cusp form for $N \ge 1$, suppose $r \in L' \cap (\overline{P} \backslash P)$. Then $Q(r) = 0$, and the growth condition on $F$ implies that $c(r,\mu) = 0$ unless $\mu = 0$. Since $G_N^{k+(1-\ell)/2}(0,0) = 0$ for $N \ge 1$ it follows that the coefficient of $\mathbf{q}^r$ in $P_N^LF(z;\lambda)$ is $0$. A similar argument applies to all cusps of $X_L$. Finally if we replace $\lambda$ by $a \cdot \lambda$ and write $$F(z+w a\lambda) = \sum_{r \in L'} \sum_{\mu \in (2m)^{-1} \mathbb{Z}} c(r,\mu) \mathbf{q}^r s^{a\mu} = \sum_{r \in L'} \sum_{\mu \in (2ma^2)^{-1}\mathbb{Z}} c(r,a^2 \mu) \mathbf{q}^r e^{2\pi i (2ma^2 \mu)aw}$$ then \begin{align*} P_N^L F(z;a\lambda) &= \sum_{r \in L'} \sum_{\mu \in (2ma^2)^{-1}\mathbb{Z}} c(r,a^2 \mu) G_N^{k+(1-\ell)/2}(2ma^3 \mu, ma^2 Q(r)) \mathbf{q}^r \\ &= \sum_{r \in L'} \sum_{\mu \in (2m)^{-1}\mathbb{Z}} c(r,\mu) G_N^{k+(1-\ell)/2}(2ma \mu, ma^2 Q(r)) \mathbf{q}^r, \end{align*} which equals $a^N P_N^L F(z;\lambda)$ because of the homogeneity $G_N^s(ax,a^2y) = a^N G_N^s(x,y)$.
\end{proof}

\begin{rem} Since $P_N^L F(z;\lambda)$ is homogeneous in $\lambda$, it is natural to extend it to $\lambda \in L^{\perp} \otimes \mathbb{C}$ by defining $$P_N^L F(z;\lambda) := a^{-N} P_N^L F(z; a\lambda)$$ where $a \in \mathbb{C}^{\times}$ is such that $a \lambda \in \Lambda$ is a lattice vector.
\end{rem}

\begin{rem} It must be emphasized that $P_N^L (F |_k A)$ cannot be expressed in terms of $P_N^L (F)$ for \emph{all} matrices $A \in \Gamma_{\Lambda}$; only those of the form $A = \tilde M$ with $M \in \Gamma_L$. In particular the higher pullbacks of Eisenstein series $$E_{k}(z) = \sum_{A \in \Gamma_{\Lambda,\infty} \backslash \Gamma_{\Lambda}} 1 \Big|_k A,$$ where $\Gamma_{\Lambda,\infty}$ is the subgroup fixing some $0$-dimensional cusp, are generally nonzero, while trivially $P_N^L(1) = 0$ for all $N \ge 1$. These pullbacks can be computed using Section 5 below because $E_k(z)$ is a theta lift.
\end{rem}

\begin{rem} Suppose $F : \mathbb{H}_{\Lambda} \rightarrow \mathbb{C}$ is an orthogonal modular form of weight $k$ which vanishes to order $N$ on $\mathbb{H}_L$, and $\lambda \in L^{\perp}$ is primitive. The \emph{quasi-pullback} of $F$ to $\mathbb{H}_L$ is (up to $\pm 1$) the limit $$\mathrm{Q} F(z) = \lim_{w \rightarrow 0} (1 - e^{2\pi i w})^{-N} F(z+w\lambda) = \lim_{w \rightarrow 0} (2\pi i w)^{-N} F(z+w\lambda)$$ and is a modular form of weight $k+N$ and a cusp form if $N > 0$. This is, up to a constant multiple, a special case of the $N$-th pullback: since $\partial_w^j |_{w=0} F = 0$ for all $j < N$ and $\partial_w^n |_{w=0} F(z) = N! \cdot \mathrm{Q}F(z)$, we find \begin{align*} P_N^L F(z;\lambda) &= (2\pi i)^{-N} \frac{N!}{\Gamma(k+\frac{1-\ell}{2} + \lfloor N/2 \rfloor)} \sum_{2n_1 + n_2 = N} (-m)^{n_1} \frac{\Gamma(k+\frac{1-\ell}{2}+n_1 + n_2)}{n_1! n_2!} \Delta_L^{n_1} \Big( \partial_w^{n_2} \Big|_{w=0} F(z) \Big) \\ &=  \frac{N! \Gamma(k+\frac{1-\ell}{2} + N)}{\Gamma(k+\frac{1-\ell}{2} + \lfloor N/2 \rfloor)}  \cdot \mathrm{Q}F(z). \end{align*} More generally, $P_{N;k}^L F(z;\lambda)$ equals $P_{N-j;k+j}^L(w^{-j} F(z+w \lambda))$ up to a constant multiple for every $0 \le j \le N$.
\end{rem}

\begin{rem} It is possible to modify the Taylor coefficients of \emph{meromorphic} modular forms $F$ to obtain pullbacks using the arguments of this section. Here the poles of $F$ are irrelevant unless they occur on the divisor $w = 0$. Therefore assume $F$ has weight $k$ and a pole of order $\ell$ along $w=0$ and write out its Taylor expansion in the form $$F(z + \lambda w) = \sum_{j = -\ell}^{\infty} \phi_j(z;\lambda) w^j.$$ For any $v \in L \otimes \mathbb{R}$ of norm $1$, since $(w^{\ell} F)|_{k-\ell} J_v = w^{\ell} (F|_k J_v) = w^{\ell} F$, Theorem 3.5 gives us the equation $$P_{N;k-\ell}^L \Big( w^{\ell} F \Big) = P_{N;k-\ell}^L \Big( (w^{\ell} F) \Big|_{k-\ell} J_v \Big) = \Big( P_{N;k-\ell}^L (w^{\ell} F) \Big) \Big|_{k+N-\ell} J_v;$$ and as before, the behavior of $P_{N;k-\ell}^L (w^{\ell} F)$ under rotations, translations and scalings of $\mathbb{H}_L$ is easy to see. In particular, the $N$-th pullback of $$w^{\ell} F = \sum_{j=0}^{\infty} \phi_{j-\ell}(z;\lambda) w^j,$$ treated as if it were a modular form of weight $k-\ell$, is a true (meromorphic) modular form of weight $k-\ell+N$.

\end{rem}

\section{Pullbacks to special cycles}

In this section we define pullback operators to special cycles of arbitrary codimension. Let $(\Lambda,Q)$ be an even lattice of signature $(1,\ell-1)$ and let $L \subseteq \Lambda$ be a Lorentzian sublattice of arbitrary rank with orthogonal complement $L^{\perp}$ in $\Lambda$.  As before the orthogonal upper half-space can be written in coordinates as $$\mathbb{H}_{\Lambda} = \{z+w: \; z \in \mathbb{H}_L, \; w \in L^{\perp} \otimes \mathbb{C}, \; Q(\mathrm{im}(w)) < Q(\mathrm{im}(z))\},$$ and $X_L$ embeds as an analytic cycle on $X_{\Lambda}$ cut out locally by the equations $w = 0$. (Here we are not necessarily assuming that $L$ has strictly lower rank than $\Lambda$; although if $\mathrm{rank}\, L = \mathrm{rank}\, \Lambda$ then everything below is vacuous.)

Any modular form $F : \mathbb{H}_{\Lambda} \rightarrow \mathbb{C}$ can be expanded as a Fourier series in the form $$F(z,w) = \sum_{r \in L'} \sum_{\mu \in (L^{\perp})'} c(r,\mu) \mathbf{q}^{r} \mathbf{s}^{\mu}, \; \; \mathbf{q}^{r} = e^{2\pi i \langle r,z \rangle}, \; \mathbf{s}^{\mu} = e^{2\pi i \langle \mu,w \rangle},$$ where $c(r,\mu) = 0$ if $(r,\mu) \notin \Lambda'$.

We will work in the dual tensor algebra $T^*L^{\perp} = \bigoplus_{n=0}^{\infty} ((L^{\perp})^*)^{\otimes n}.$ Given a linear form $r : L^{\perp} \rightarrow \mathbb{Z}$ and a bilinear form $m : L^{\perp} \times L^{\perp} \rightarrow \mathbb{Z}$ we define multilinear $N$-forms $G_N^s(r,m) \in ((L^{\perp})^*)^{\otimes N}$ by analogy to the usual Gegenbauer polynomials. Namely we define $g_N^s(r,m)$ by symmetrizing the coefficient of $t^N$ in $(1 -rt + mt^2)^{-s} \in (T^*L^{\perp})[|t|]$, and we define $$G_N^s(r,m) := \frac{N! \Gamma(s)}{\Gamma(s + \lceil N/2 \rceil)} g_N^s(r,m).$$ For example $$G_0^s(r,m) = 1, \; G_1^s(r,m) = r, \; G_2^s(r,m) = (s+1) r \otimes r - 2\mathrm{Sym}(m).$$ Here the symmetrization of an $N$-form $\omega$ is $\mathrm{Sym}(\omega)(v_1,...,v_n) = \frac{1}{N!} \sum_{\sigma \in S_n} \omega(v_{\sigma(1)},...,v_{\sigma(n)})$.

\begin{thm}\label{pullbacks} Let $F : \mathbb{H}_{\Lambda} \rightarrow \mathbb{C}$ be an orthogonal modular form of weight $k$ with Fourier series $$F(z,w) = \sum_{\lambda \in L'} \sum_{\mu \in (L^{\perp})'} c(\lambda,\mu) \mathbf{q}^{\lambda} \mathbf{s}^{\mu}, \; \; \mathbf{q}^{\lambda} = e^{2\pi i \langle \lambda,z \rangle}, \; \mathbf{s}^{\mu} = e^{2\pi i \langle \mu,w \rangle}.$$ Let $B$ denote the bilinear form $B(x,y) = Q(x+y) - Q(x) - Q(y)$ restricted to $L^{\perp}$ and interpret dual lattice vectors $\mu \in (L^{\perp})'$ as the linear forms $v \mapsto \langle v, \mu \rangle$. For $N \in \mathbb{N}_0$, $$P_N^{L} F(z) = \sum_{\lambda \in L'} \Big( \sum_{\mu \in (L^{\perp})'} c(\lambda,\mu) G_N^{k - \mathrm{dim}\, L/2}(\mu, -Q(\lambda)B/2) \Big) \mathbf{q}^{\lambda}$$ is a $T^* L^{\perp}$-valued orthogonal modular form of weight $k+N$ on $\mathbb{H}_L$.
\end{thm}
More explicitly, for any vectors $v_1,...,v_N \in L^{\perp} \otimes \mathbb{C}$ we obtain an orthogonal modular form $$P_N^L F(z;v_1,...,v_N) = \sum_{\lambda \in L'} \Big( \sum_{\mu \in (L^{\perp})'} c(\lambda,\mu) G_N^{k - \mathrm{dim}\, L/2}(\mu,-Q(\lambda)B/2)(v_1,...,v_N) \Big) \mathbf{q}^{\lambda}.$$
\begin{proof} Fix any vector $v \in L^{\perp}$. By restricting $F$ to the symmetric space associated to the lattice $L \oplus \langle v \rangle$  (interpreted as a subset of $\mathbb{H}_L \times \mathbb{C}$) we obtain the modular form $$F_v : \mathbb{H}_{L \oplus \langle v \rangle} \rightarrow \mathbb{C}, \; \; F_v(z,w) := F(z + wv) = \sum_{\lambda \in L'} \sum_{\mu \in (L^{\perp})'} c(\lambda,\mu) \mathbf{q}^{\lambda} s^{\mu(v)}, \; \; s = e^{2\pi i w},$$ whose higher pullbacks $P_N^L(F_v)$ are modular forms of weight $k+N$ on $\mathbb{H}_L$. The Fourier expansion of $P_N^L(F_v)$ is $$P_N(F_v)(z) = \sum_{\lambda \in L'} \Big( \sum_{\mu \in (L^{\perp})'} c(\lambda,\mu) G_N^{k - \mathrm{dim}\, L / 2}(\mu(v), -Q(\lambda)Q(v)) \Big) \mathbf{q}^{\lambda},$$ i.e. $P_N(F_v)(z) = P_N^{L} F(z;v,...,v)$. In particular $P_N^{L} F$ is a symmetric multilinear form whose diagonal values $P_N^L F(z;v,...,v)$ are modular forms, and cusp forms if $N \ge 1$. Every value $P_N^L F(z;v_1,...,v_n)$ is obtained as a linear combination of the diagonal values through the polarization identity.
\end{proof}

\section{Partial development coefficients of Jacobi forms and higher pullbacks of theta lifts}

One important construction of orthogonal modular forms is the theta lift from vector-valued modular forms (of half-integral weight) for $\mathrm{Mp}_2(\mathbb{Z})$. It turns out that the (higher) pullbacks of theta lifts are themselves theta lifts. In the special case that our Lorentzian lattice $\Lambda$ splits in the form $L(-1) \oplus \mathrm{II}_{1,1}$, one can identify the vector-valued modular forms in question with Jacobi forms of lattice index $L$, and the pullbacks of the theta lift of a Jacobi form is essentially the theta lift of its development coefficients (in an appropriate sense). In this section we explain how the development coefficients are defined for Jacobi forms of lattice index and we explain how to generalize this to vector-valued modular forms attached to lattices of arbitrary signature. Finally we prove that the development coefficients and higher pullbacks fit into a commutative diagram involving the theta lift. \\

\subsection{Jacobi forms} Let $(L,Q)$ be a positive-definite even lattice and let $\rho : \mathrm{Mp}_2(\mathbb{Z}) \rightarrow \mathrm{GL}\, V$ be a finite-dimensional representation. A \emph{Jacobi form} of weight $k$ and index $L$ and multiplier $\rho$ is a holomorphic function $\phi : \mathbb{H} \times (L \otimes \mathbb{C}) \rightarrow V$ satisfying a vanishing condition on Fourier coefficients (explained below) and the following functional equations. \\
(i) For all $(\begin{psmallmatrix} a & b\\ c &d \end{psmallmatrix}, \phi) \in \mathrm{Mp}_2(\mathbb{Z})$, $$\phi \left( \frac{a \tau + b}{c \tau + d}, \frac{z}{c \tau + d} \right) = \phi(\tau)^{2k} e^{2\pi i \frac{c}{c \tau + d} Q(z)} \rho(\begin{psmallmatrix} a & b \\ c & d \end{psmallmatrix}, \phi) \phi(\tau,z).$$
(ii) For all $\lambda,\mu \in L$, $$\phi(\tau,z+\lambda \tau) = e^{-2\pi i (\tau Q(\lambda) + \langle \lambda,z \rangle)} \phi(\tau,z) \; \text{and} \; \phi(\tau,z+\mu) = \phi(\tau,z).$$

Every Jacobi form $\phi$ has a Fourier expansion which takes the form $$\phi(\tau,z) = \sum_{n \in \mathbb{Q}} \sum_{r \in L'} c(n,r) q^n \zeta^r, \; q = e^{2\pi i \tau}, \; \zeta^r = e^{2\pi i \langle r,z \rangle}, \; c(n,r) \in V.$$ The vanishing condition referred to above is that $c(n,r) = 0$ whenever $Q(r) > n$.

\begin{rem} The more familiar (scalar-valued) Jacobi forms of weight $k$ and index $m \in \mathbb{N}$ as in \cite{EZ} are the special case of Jacobi forms indexed by a rank-one lattice: $L = \mathbb{Z}$ and $Q(v) = mv^2$ with bilinear form $\langle v,w \rangle = 2mvw$. Also, modular forms are Jacobi forms for the rank-zero lattice $L = \{0\}$.
\end{rem}

Given a Jacobi form $\phi$ of weight $k$, the development coefficients $D_0 \phi, D_1 \phi$ are a modification of the Taylor coefficients of $\phi$ about $z=0$ which yield vector-valued modular forms of weight $k,k+1,...$ taking values in $V \otimes T^*L$, where $T^*L = \bigoplus_{n=0}^{\infty} (L^* \otimes \mathbb{C})^{\otimes n}$. In particular $D_N \phi$ should be thought of as a symmetric multilinear form which takes $N$ input vectors in $L$ and produces a modular form of weight $k+N$ and multiplier $\rho$.

\begin{lem} Let $B(x,y) = \langle x,y \rangle = Q(x+y) - Q(x) - Q(y)$ be the bilinear form on $L$ and identify dual lattice vectors $r$ with one-forms $v \mapsto \langle r,v \rangle$. Let $\phi(\tau,z) = \sum_{n,r} c(n,r) q^n \zeta^r$ be a Jacobi form of weight $k$ and index $L$ and multiplier $\rho$. For any $N \in \mathbb{N}_0$ and $\alpha \in L^{\otimes N}$, $$D_N \phi(\tau;\alpha) = \sum_{n \in \mathbb{Q}} \Big( \sum_{r \in L'} c(n,r) G_N^{k-1}(r,nB/2)(\alpha) \Big)q^n$$ is a modular form of weight $k+N$ and multiplier $\rho$ and it is a cusp form if $N > 0$.
\end{lem}
Here $G_N^{k-1}$ are the multilinear Gegenbauer polynomials as defined in the previous section. For example the first few development coefficients of $\phi$ are \begin{align*} D_0 \phi(\tau) &= \sum_n \Big( \sum_{r \in L'} c(n,r) \Big)q^n; \\ D_1 \phi(\tau;v) &= \sum_n \Big( \sum_{r \in L'} \langle r,v \rangle c(n,r) \Big) q^n; \\ D_2 \phi(\tau;v_1 \otimes v_2) &=  \sum_n \Big( \sum_{r \in L'} (k \langle r,v_1 \rangle \langle r, v_2 \rangle - n \langle v_1, v_2 \rangle) c(n,r) \Big) q^n. \end{align*}
\begin{proof} As in the proof of Theorem 4.1, use the polarization identity to reduce to the case that $\alpha$ is a diagonal pure tensor: $\alpha = v \otimes ... \otimes v$ for some $v \in L$. The form $$\phi_v : \mathbb{H} \times \mathbb{C} \rightarrow V, \; \phi_v(\tau,z) = \phi(\tau,zv)$$ is a Jacobi form of the same weight and multiplier and of scalar index $Q(v)$, and $D_N \phi(\tau;\alpha)$ is, up to a scalar multiple, the $N^{\text{th}}$ development coefficient of $\phi_v$ in the sense of Chapter 3 of \cite{EZ}.
\end{proof}

\subsection{Weil representations and partial development coefficients} Let $(\Lambda,Q)$ be an even lattice with discriminant group $A = \Lambda'/\Lambda$, and let $\sigma$ be its signature. The \emph{Weil representation} associated to $\Lambda$ is the representation $\rho_{\Lambda}$ of $\mathrm{Mp}_2(\mathbb{Z})$ on the group ring $\mathbb{C}[A]$ where the generators $$S = \left( \begin{psmallmatrix} 0 & -1 \\ 1 & 0 \end{psmallmatrix}, \sqrt{\tau} \right), \; T = \left( \begin{psmallmatrix} 1 & 1 \\ 0 & 1 \end{psmallmatrix}, 1 \right)$$ act on the natural basis $\mathfrak{e}_{\gamma}$, $\gamma \in A$ by $$\rho_{\Lambda}(T) \mathfrak{e}_{\gamma} = e^{2\pi i Q(\gamma)} \mathfrak{e}_{\gamma}, \; \; \rho_{\Lambda}(S) \mathfrak{e}_{\gamma} = \frac{1}{\sqrt{|\Lambda'/\Lambda|}} e^{2\pi i (\sigma /8)} \sum_{\beta \in \Lambda'/\Lambda} e^{-2\pi i \langle \gamma, \beta \rangle} \mathfrak{e}_{\beta}.$$

The most familiar setting where this representation appears is the multiplier of the theta function. If $\Lambda$ is positive-definite then $$\Theta_{\Lambda}(\tau,z) = \sum_{\lambda \in \Lambda'} q^{Q(\lambda)} \zeta^{\lambda} \mathfrak{e}_{\lambda + \Lambda}$$ is a Jacobi form of weight $(\mathrm{dim}\, \Lambda)/2$, lattice index $\Lambda$, and multiplier $\rho_{\Lambda}$. \\


Suppose $L \subseteq \Lambda$ is a sublattice and define $\tilde \Lambda := L^{\perp} \oplus L \subseteq \Lambda$ with discriminant group $B$. Then the Weil representations attached to $L$ and $\tilde \Lambda$ (or $A$ and $B$) are related by the intertwining operators $$\downarrow^{A}_{B} : \mathbb{C}[A] \longrightarrow \mathbb{C}[B], \; \; \mathfrak{e}_{\gamma} \mapsto \sum_{\substack{\delta \in \tilde \Lambda'/\tilde \Lambda \\ \delta + \Lambda = \gamma}} \mathfrak{e}_{\delta}$$ and $$\uparrow^{A}_{B} : \mathbb{C}[B] \longrightarrow \mathbb{C}[A], \; \; \mathfrak{e}_{\delta} \mapsto \begin{cases} \mathfrak{e}_{\delta + \Lambda}: & \delta \in \Lambda'; \\ 0: & \text{otherwise}; \end{cases}$$ in the sense that $\rho_{\tilde \Lambda}(M) \downarrow^{A}_{B} = \downarrow^{A}_{B} \rho_{\Lambda}(M)$ and $\rho_{\Lambda}(M) \uparrow^{A}_{B}= \uparrow^{A}_{B} \rho_{\tilde \Lambda}(M)$ for every $M \in \mathrm{Mp}_2(\mathbb{Z})$, as one can check on the generators $M = S,T$. On modular forms these induce the up-arrow and down-arrow maps as in Lemma 5.5 and Lemma 5.6 of \cite{Br}.

We define a  $\mathbb{C}$-linear ``trace map" along the sublattice $L \subseteq \Lambda$ by $$\mathrm{Tr}_L : \mathbb{C}[L'/L] \otimes \mathbb{C}[\Lambda'/\Lambda] \otimes \mathbb{C}[\Lambda'/\Lambda] \longrightarrow \mathbb{C}[L'/L], \; \; \mathfrak{e}_{\beta} \otimes \mathfrak{e}_{\gamma} \otimes \mathfrak{e}_{\delta} \mapsto \begin{cases} \mathfrak{e}_{\beta}: & \gamma = \delta; \\ 0 : & \text{otherwise}. \end{cases}$$ This respects the Weil representations in the sense that $$\mathrm{Tr}_L \Big( \rho_L(M) \mathfrak{e}_{\beta} \otimes \rho_{\Lambda}(M) \mathfrak{e}_{\gamma} \otimes \rho_{\Lambda(-1)}(M) \mathfrak{e}_{\delta}\Big) = \rho_L(M) \mathrm{Tr}_L(\mathfrak{e}_{\beta} \otimes \mathfrak{e}_{\gamma} \otimes \mathfrak{e}_{\delta})$$ for all $M \in \mathrm{Mp}_2(\mathbb{Z}).$ In the special case that $L = \{0\}$ we simply write $$\mathrm{Tr} : \mathbb{C}[\Lambda'/\Lambda] \otimes \mathbb{C}[\Lambda'/\Lambda] \longrightarrow \mathbb{C}.$$

The trace map may be used to define the \emph{theta decomposition} which is an isomorphism $\Theta$ between certain spaces of vector-valued modular forms and Jacobi forms. Specifically, if $\Lambda$ is a negative-definite even lattice then we get isomorphisms $$\Theta : M_k(\rho_{\Lambda}) \stackrel{\sim}{\rightarrow} J_{k+\mathrm{dim}\, \Lambda/2,\Lambda(-1)}, \; \; \Theta F = \mathrm{Tr}(F \otimes \Theta_{\Lambda(-1)}).$$ On Fourier expansions this acts formally by sending $q^n \mathfrak{e}_{\lambda + \Lambda}$ to $q^{n - Q(\lambda)} \zeta^{\lambda}$; in other words, $$\Theta \Big( \sum_{\gamma \in \Lambda'/\Lambda} \sum_{n \in \mathbb{Z} + Q(\gamma)} c(n,\gamma) q^n \mathfrak{e}_{\gamma} \Big) = \sum_{\lambda \in \Lambda'} \sum_{n \in \mathbb{Z} + Q(\lambda)} c(n,\lambda + \Lambda) q^{n - Q(\lambda)} \zeta^{\lambda}.$$

These operators may be used to define the (partial) development coefficients of a modular form $F \in M_k(\rho_{\Lambda})$ along a negative-definite sublattice of $\Lambda$.

\begin{defn} Let $(\Lambda,Q)$ be an even lattice and let $L \subseteq \Lambda$ be a sublattice with negative-definite orthogonal complement $L^{\perp}$. Let $F \in M_k(\rho_{\Lambda})$. For $N \in \mathbb{N}_0$ the \emph{development coefficients of $F$ along $L$} are the development coefficients of a corresponding Jacobi form: $$D^{L}_N F(\alpha) := D_N \mathrm{Tr}_{L}\Big((\downarrow^{\Lambda}_{L^{\perp} \oplus L} F) \otimes \Theta_{L^{\perp}(-1)} \Big)(\alpha) \in M_{k+N+\mathrm{dim}\, K / 2}(\rho_L), \; \alpha \in (L^{\perp})^{\otimes N}.$$ Here $\mathrm{Tr}_{L}(\downarrow^{\Lambda}_{L^{\perp} \oplus L} F \otimes \Theta_{L^{\perp}(-1)})$ is a Jacobi form of weight $k+\mathrm{dim}\, L^{\perp}/2$, index $L^{\perp}(-1)$ and multiplier $\rho_{L}$.
\end{defn}

In particular, if $F$ has Fourier expansion $F(\tau) = \sum_{\gamma \in \Lambda'/\Lambda} \sum_{n \in \mathbb{Z} + Q(\gamma)} c(n,\gamma) q^n \mathfrak{e}_{\gamma}$, then its development coefficients along $L$ have Fourier expansions $$D_N^L F(\tau;\alpha) = \sum_{\gamma \in L'/L} \sum_{n \in \mathbb{Z} + Q(\gamma)} \Big( \sum_{\substack{\lambda \in (L^{\perp})' \\ (\gamma,\lambda+L^{\perp}) \in \Lambda'/\Lambda}} c(n+Q(\lambda),(\gamma,\lambda)) \cdot G_N^{k - 1 + \mathrm{dim}\, L^{\perp}/2}(\lambda, n B/2)(\alpha) \Big) q^n \mathfrak{e}_{\gamma}$$ for $\alpha \in (L^{\perp})^{\otimes n}.$

\begin{cor} Let $L$ be a positive-definite even lattice and let $K \subseteq L$ be a sublattice with orthogonal complement $K^{\perp}$ in $L$. Let $\phi(\tau,z) = \sum_{n,r} c(n,r) q^n \zeta^r$ be a Jacobi form of weight $k$, index $L$ and multiplier $\rho$. Let $B$ denote the bilinear form on $K^{\perp}$ induced by restricting $Q$. For every $N \in \mathbb{N}_0$ and $\alpha \in (K^{\perp})^{\otimes N}$, the \emph{partial development coefficient} $$D_N^K(\tau,z;\alpha) := \sum_{n \in \mathbb{Q}} \sum_{r_K \in K'} \Big( \sum_{r_{K^{\perp}} \in (K^{\perp})'} c(n,r_K,r_{K^{\perp}}) G_N^{k-1-(\mathrm{dim}\, K)/2}(r_{K^{\perp}},\frac{nB}{2})(\alpha) \Big) q^n \zeta^{r_K}, \; z \in K \otimes \mathbb{C}$$ is a Jacobi form of weight $k+N$, index $K$ and multiplier $\rho$.
\end{cor}
As usual $c(n,r_K,r_{K^{\perp}}) = c(n,r)$ if there is a linear form $r : L \rightarrow \mathbb{Z}$ with $r|_K = r_K$ and $r|_{K^{\perp}} = r_{K^{\perp}}$ and $c(n,r_K,r_{K^{\perp}}) = 0$ otherwise.
\begin{proof} If $\phi$ has theta decomposition $F$ then $D_N^K \phi(\tau,z;\alpha)$ is the Jacobi form whose theta decomposition is $D_N^{K(-1)} F(\tau;\alpha)$.
\end{proof}

\subsection{Fourier-Jacobi expansions} Suppose $\Lambda$ splits in the form $L(-1) \oplus \mathrm{II}_{1,1}$ where $L$ is positive-definite and write the upper half-space in the form $$\mathbb{H}_{\Lambda} = \Big\{(\tau,z,w): \;\tau,z \in \mathbb{H}, \; z \in L \otimes \mathbb{C}, \; Q_L(\mathrm{im}\, z) < (\mathrm{im}\, \tau) \cdot (\mathrm{im}\, w) \Big\}.$$ The partial Fourier expansion of an orthogonal modular form $F$ for $\Lambda$ with respect to $w$ takes the form $$F(\tau,z,w) = \sum_{n=0}^{\infty} \phi_n(\tau,z) s^n, \; s = e^{2\pi i w}.$$ Each $\phi_n$ above can be shown to be a Jacobi form of weight $k$ and lattice index $L(n)$, i.e. $L$ with quadratic form $n \cdot Q$ (cf. \cite{G} for details). The Fourier expansion above is called the \emph{Fourier-Jacobi expansion} of $F$. \\

Certain pullbacks have the simple effect of applying partial development coefficients to the Fourier-Jacobi expansion termwise:

\begin{prop} Suppose $F$ is an orthogonal modular form of weight $k$ with Fourier-Jacobi expansion $$F(\tau,z,w) = \sum_{n=0}^{\infty} \phi_n(\tau,z) s^n.$$ Let $K \subseteq L$ be a sublattice. Then the pullbacks of $F$ to $K(-1) \oplus \mathrm{II}_{1,1}$ have Fourier-Jacobi expansions $$P_N^{K(-1) \oplus \mathrm{II}_{1,1}} F(\tau,z_1,w;\alpha) =  \sum_{n=0}^{\infty} D_N^K \phi_n(\tau,z_1;\alpha) s^n, \; (\tau,z_1,w) \in \mathbb{H}_{K(-1) \oplus \mathrm{II}_{1,1}}.$$
\end{prop}
\begin{proof} It is enough to prove this when $K$ has codimension $1$, since when $\alpha = v \otimes ... \otimes v$ is a pure diagonal we obtain both sides by first restricting from $L$ to $K \oplus \langle v \rangle$, and then pulling back or applying development coefficients to $K$.
Write out each term in the Fourier-Jacobi expansion as $$\phi_n(\tau,z) = \sum_{j=0}^{\infty} \sum_{r \in L'} c(j,r,n) q^j \zeta^r, \; q = e^{2\pi i \tau}, \; \zeta^r = e^{2\pi i r(z)}.$$ Fix $\lambda \in K^{\perp}$ of norm $Q_L(\lambda) = m$ and decompose $z = z_1 + z_2 \lambda$ where $z_1 \in K \otimes \mathbb{C}$. Write $\zeta_1^r = e^{2\pi i r(z_1)}$ for $r \in K'$. Since $k - 1 - (\mathrm{dim}\, K)/2 = k + (1 - \ell)/2$ where $\ell = \mathrm{dim}\, \Lambda$ we find \begin{align*} P_N^{K(-1) \oplus \mathrm{II}_{1,1}} F(\tau,z_1,w;\lambda) &= \sum_{n=0}^{\infty} \Big[ \sum_{j=0}^{\infty} \sum_{r \in K'} \Big( \sum_{\mu \in (2m)^{-1}\mathbb{Z}} c(j,(r,\mu),n) G_N^{k+(1-\ell)/2}(-2m \mu, -mQ_L(r)) \Big) q^j (\zeta_1)^r \Big] s^n \\ &= \sum_{n=0}^{\infty} D_N^K \phi(\tau,z_1) s^n. \qedhere \end{align*}
\end{proof}

\subsection{Theta lifts} Suppose $\Lambda$ is a Lorentzian lattice of signature $(1,\ell-1)$. If $$F(\tau) = \sum_{\gamma \in \Lambda'/\Lambda} \sum_{n \in \mathbb{Z} + Q(\lambda)} c(n,\gamma)q^n \mathfrak{e}_{\gamma}$$ is a cusp form of weight $k + 1 - \ell/2$ with $k \ge 2$, then the \emph{theta lift} $$\Phi_F(z) = \sum_{\lambda \in \Lambda' \cap P} \sum_{n=1}^{\infty} c(Q(\lambda),\lambda) n^{k-1} \mathbf{q}^{n \lambda}, \; z \in \mathbb{H}_{\Lambda}$$ is an orthogonal cusp form of weight $k$. (See \cite{Bn}, \cite{Br} for the approach to the theta lift using vector-valued modular forms; for the approach which takes Jacobi forms as input we refer to Gritsenko, e.g. \cite{GN}.)

\begin{prop} Suppose $F \in S_{k+1-\ell/2}(\rho_{\Lambda})$ is a cusp form of weight $k+1-\ell/2$ with $k \ge 2$ and $L \subseteq \Lambda$ is a negative-definite sublattice. Then the $N^{th}$ pullback of the theta lift $\Phi_F$ along $L$ equals the theta lift of the $N^{th}$ development coefficient of $F$ along $L$: $$P_N^L \Phi_F = \Phi_{D_N^L F}.$$
\end{prop}
\begin{proof} Let $F$ have the Fourier expansion $$F(\tau) = \sum_{\gamma \in \Lambda'/\Lambda} \sum_{n \in \mathbb{Z} + Q(\gamma)} c(n,\gamma) q^n \mathfrak{e}_{\gamma}.$$ Then $D_N^L F$ is a modular form of weight $k+1-\ell/2+N + \mathrm{dim}\, L/2 = k+1+N - \mathrm{dim}\, L^{\perp}/2$ with Fourier expansion $$D_N^L F(\tau; \alpha) = \sum_{\gamma \in (L^{\perp})' / L^{\perp}} \sum_{n \in \mathbb{Z} + Q(\gamma)} \Big( \sum_{\substack{\lambda \in L' \\ (\gamma,\lambda) \in \Lambda'}} c(n+Q(\lambda),(\gamma,\lambda)) G_N^{k - \mathrm{dim}\, L^{\perp}/2}(\lambda, nB/2)(\alpha) \Big) q^n \mathfrak{e}_{\gamma}$$ and its theta lift is \begin{align*} \Phi_{D_N^L F(\alpha)} &= \sum_{\mu \in (L^{\perp})'} \sum_{n=0}^{\infty} n^{k+N-1} \sum_{\substack{\lambda \in L' \\ (\mu,\lambda) \in \Lambda'}} c\Big( Q(\mu) + Q(\lambda), (\mu,\lambda) \Big) G_N^{k - \mathrm{dim}\, L^{\perp}/2}(\lambda, Q(\mu) B / 2)(\alpha) \mathbf{q}^{n \lambda} \\ &= \sum_{\mu \in (L^{\perp})'} \sum_{n=0}^{\infty} n^{k-1} \sum_{\substack{\lambda \in L' \\ (\mu,\lambda) \in \Lambda'}} c \Big( Q(\mu)+Q(\lambda),(\mu,\lambda) \Big) G_N^{k-\mathrm{dim}\, L^{\perp}/2}(n \lambda, Q(n\mu)B/2)(\alpha) \mathbf{q}^{n\lambda}, \end{align*} i.e. the $N^{\text{th}}$ pullback of $$\Phi_F = \sum_{\lambda \in \Lambda'} \sum_{n=0}^{\infty} n^{k-1} c(Q(\lambda),\lambda) \mathbf{q}^{n \lambda}$$ along $L$ evaluated at $\alpha \in L^{\otimes N}$.
\end{proof}

\section{Special cases}

In this section we work out the pullbacks explicitly in three cases. These serve as examples of how these calculations are carried out in general and also may be of independent interest. \\

\subsection{Hilbert modular forms} Let $\mathcal{O}$ be an order in a real-quadratic field $K$, which can be understood as an even Lorentzian lattice with respect to the conjugate-trace form $\langle x, y \rangle = \mathrm{Tr}_{K/\mathbb{Q}}(xy')$ and quadratic form $Q(x) = xx' = N_{K/\mathbb{Q}}(x)$. Here and below, $x'$ denotes the conjugate of $x \in K$. Let $\mathcal{O}^{\#}$ be the dual lattice and let $d_{\mathcal{O}} = |\mathcal{O}^{\#}/\mathcal{O}|$ be the discriminant. (This simplifies for quadratic fields: if $d_K$ is the discriminant of $K$ then $\mathcal{O}$ must have the form $$\mathcal{O} = \mathbb{Z}\Big[f \cdot \frac{d_K + \sqrt{d_K}}{2}\Big], \; \text{where} \; f := [\mathcal{O}_K : \mathcal{O}],$$ and $d_{\mathcal{O}} = f^2 d_K$, and $\mathcal{O}^{\#} = \frac{1}{\sqrt{d_{\mathcal{O}}}} \mathcal{O}.$ See e.g. \cite{C} 7.A.) Then $\mathcal{O} \oplus \mathrm{II}_{1,1}$ is isometric to the lattice of conjugate-skew-symmetric matrices $\begin{psmallmatrix} a \sqrt{d_{\mathcal{O}}} & -b \\ b' & c / \sqrt{d_{\mathcal{O}}} \end{psmallmatrix}$ with $a,b,c \in \mathcal{O}$, which is acted upon by the \emph{Hilbert modular group} $$\Gamma_{\mathcal{O}} = \mathrm{PSL}_2(\mathcal{O}^{\#} \oplus \mathcal{O}) = \Big\{ \begin{psmallmatrix} a & b \\ c & d \end{psmallmatrix} \in \mathrm{PSL}_2(K): \; a,d \in \mathcal{O}, \, b \in \mathcal{O}^{\#},\, c \in (\mathcal{O}^{\#})^{-1} \Big\}$$ by conjugation $A \cdot M = A'MA^T$. (Here $(\mathcal{O}^{\#})^{-1}$ is the inverse in the group of fractional ideals of $K$.)

\begin{rem} The group $\mathrm{PSL}_2(\mathcal{O})$ (which is also often called the Hilbert modular group) appears in the orthogonal modular group attached to $\mathcal{O}$ with the \emph{negative} norm-form $Q(x) = -xx'$. Jordan decomposition shows that the discriminant forms ($\mathcal{O}^{\#}/\mathcal{O}, \pm Q)$ are equivalent if and only if $d_K$ is not divisible by any prime $p \equiv 3 \, (4)$. (It is exactly in these cases that the associated surfaces $\mathrm{PSL}_2(\mathcal{O}) \backslash (\mathbb{H} \times \mathbb{H}) $ and $\Gamma_{\mathcal{O}} \backslash (\mathbb{H} \times \mathbb{H})$ are isomorphic, as observed in \cite{H}.) If $\mathcal{O}_K$ contains a unit $\varepsilon$ of norm $-1$ (e.g. if $K = \mathbb{Q}(\sqrt{p}))$ for a prime $p \equiv 1 \, (4)$) then $x \mapsto \varepsilon x$ preserves the order $\mathcal{O}$ and gives an isometry on the level of lattices. The group $\Gamma_{\mathcal{O}}$ is natural in the moduli interpretation through which points of $\Gamma_{\mathcal{O}} \backslash (\mathbb{H} \times \mathbb{H})$ correspond to principally polarized abelian surfaces with special endomorphisms.
\end{rem}

One can interpret orthogonal modular forms for $\mathcal{O}$ as graded-symmetric Hilbert modular forms of the same (parallel) weight in a way that respects Fourier expansions. One of the positive cones for $\mathcal{O}$ can be identified with $\mathbb{H} \times \mathbb{H}$ by associating, to a modulus $z \in \mathbb{H}_{\mathcal{O}} \subseteq \mathcal{O} \otimes_{\mathbb{Z}} \mathbb{C}$ (which is most conveniently thought of as the $\mathbb{Z}$-linear map $\langle -,z \rangle : \mathcal{O}^{\#} \rightarrow \mathbb{C}$) the point $(\tau_1,\tau_2) \in \mathbb{H} \times \mathbb{H}$ which satisfies $$\nu' \tau_1 + \nu \tau_2 = \langle \nu/\sqrt{d_{\mathcal{O}}}, z \rangle \; \text{for all} \; \nu \in \mathcal{O}.$$ Proving that this is equivariant with respect to the action of $\Gamma_K$ on $\mathbb{H} \times \mathbb{H}$ and its orthogonal action on $\mathbb{H}_{\mathcal{O}}$ reduces to showing that the map $$\phi : \mathbb{H} \times \mathbb{H} \longrightarrow (\mathcal{O} \oplus \mathrm{II}_{1,1}) \otimes \mathbb{C}, \; \; \phi(\tau_1,\tau_2) = d_{\mathcal{O}}^{-1/2} \begin{psmallmatrix} \tau_1 \tau_2 & \tau_2 \\ \tau_1 & 1 \end{psmallmatrix}$$ satisfies $\phi(A \cdot (\tau_1,\tau_2)) = (\gamma \tau_1 + \delta)(\gamma' \tau_2 + \delta') A' \phi(\tau_1,\tau_2) A^T$ for every $A = \begin{psmallmatrix} \alpha & \beta \\ \gamma & \delta \end{psmallmatrix} \in \mathrm{PSL}_2(\mathcal{O})$. In particular if $$F(z) = \sum_{\nu \in \mathcal{O}^{\#}} c(\nu) \mathbf{q}^{\nu}, \; \mathbf{q}^{\nu} = e^{2\pi i \langle \nu, z \rangle}, \; z \in \mathbb{H}_{\mathcal{O}} \subseteq \mathcal{O} \otimes_{\mathbb{Z}} \mathbb{C}$$ is an orthogonal modular form then $f(\tau_1,\tau_2) = \sum_{\nu \in \mathcal{O}} c(\nu/\sqrt{d_K}) e^{2\pi i (\nu' \tau_1 + \nu \tau_2)}$ is a Hilbert modular form satisfying $f(\tau_2,\tau_1) = (-1)^k f(\tau_1,\tau_2).$ \\

Suppose $\lambda,\mu \in \mathcal{O}$ are nonzero vectors such that $\langle \lambda, \mu \rangle = 0$ and $\lambda$ is totally positive. Then the rational quadratic divisor on $\mathbb{H}_{\mathcal{O}}$ associated to the (negative-norm) vector $\mu$ corresponds to the Hirzebruch-Zagier curve $\{(\lambda \tau, \lambda' \tau): \; \tau \in \mathbb{H}\} \subseteq \mathbb{H} \times \mathbb{H}$ which embeds $X_0(N_{K/\mathbb{Q}} \lambda)$ in $\overline{\Gamma_{\mathcal{O}} \backslash (\mathbb{H} \times \mathbb{H})}$. The higher pullbacks from Hilbert modular forms to elliptic modular forms of level $\Gamma_0(N_{K/\mathbb{Q}} \lambda)$ obtained in this way are essentially Cohen's operators \cite{Co}: for a Hilbert modular form $f(\tau_1,\tau_2)$ of parallel weight $(k,k)$, define $$C_N^{\lambda} f(\tau) := (2\pi i)^{-N} \sum_{r=0}^N (-1)^r \binom{k+N-1}{r} \binom{k+N-1}{N-r} \lambda^r (\lambda')^{n-r} \Big( \frac{\partial^N}{\partial \tau_1^r \partial \tau_2^{N-r}} f \Big)(\lambda \tau, \lambda' \tau).$$

\begin{prop} Let $F(z) = \sum_{\nu \in \mathcal{O}^{\#}} c(\nu) \mathbf{q}^{\nu}$ be an orthogonal modular form of weight $k$ for $\mathcal{O}$ with corresponding Hilbert modular form $f(\tau_1,\tau_2) = \sum_{\nu \in \mathcal{O}} c(\nu/\sqrt{d_K}) e^{2\pi i (\nu' \tau_1 + \nu \tau_2)}$. Then its $N^{th}$ pullback to the curve $(\lambda \tau, \lambda' \tau)$ is, up to a nonzero multiple (which depends on $k$ and $N$ and $\lambda$), the $N^{th}$ Cohen operator: \begin{align*} P_N^{(\lambda,\lambda')\mathbb{H}} F(\tau;\mu) &=  \sum_{\nu \in \mathcal{O}^{\#}} c(\nu) G_N^{k-1/2}\Big(-\langle \nu, \mu \rangle, -\frac{Q(\mu)}{4Q(\lambda)} \langle \nu, z_{\lambda} \rangle^2 \Big) q^{\langle \nu, \lambda/\sqrt{d_{\mathcal{O}}} \rangle} \\ &= \frac{N!}{(k+\lfloor N/2 \rfloor) ... (k+N-1)} (\mu / \lambda)^N C_N^{\lambda} f(\tau). \end{align*}
\end{prop}
Here we let $z_{\lambda} \in \mathbb{H}_{\mathcal{O}}$ denote the point corresponding to $(\lambda \tau, \lambda' \tau) \in \mathbb{H} \times \mathbb{H}$.
\begin{proof} To extract the first expression for $P_N F$ from the formula of section 3, consider that the coordinates of $\nu \in \mathcal{O}^{\#}$ with respect to the orthogonal splitting $\lambda,\mu$ of $\mathcal{O} \otimes \mathbb{C}$ are the projections $\frac{\langle \nu, \lambda \rangle}{2Q(\lambda)}, \frac{\langle \nu, \mu \rangle}{2Q(\mu)}$ and that $m = -Q(\mu)$. The rest of the proof relies on the fact that Cohen's coefficients $\binom{k+N-1}{r} \binom{k+N-1}{N-r}$, after rescaling by the factor $\frac{\Gamma(2k+N-1)}{\Gamma(k+N)}$, have a simple generating function which is similar to that of the Gegenbauer polynomials:

\begin{lem} As formal power series in $x$ and $y$, for any $k > 1/2$, $$\sum_{N=0}^{\infty} \frac{\Gamma(2k+N-1)}{\Gamma(k+N)} \sum_{r+s=N} \binom{k+N-1}{r} \binom{k+N-1}{s} x^r y^s = \frac{\Gamma(2k-1)}{\Gamma(k)} \Big( 1 - 2 (x+y) + (x-y)^2 \Big)^{1/2 - k}.$$

\end{lem}
\begin{proof} Once this identity has been conjectured (e.g. by comparing the pullbacks with Cohen's operators), it is straightforward to verify algebraically: for example, labelling $$F(x,y,k) := \sum_{r,s \ge 0} \binom{k+r+s-1}{r} \binom{k+r+s-1}{s} \frac{\Gamma(2k+r+s-1)}{\Gamma(k+r+s)} x^r y^s,$$ it is enough to verify that $F$ satisfies the differential equation (in $y$) $$\Big(1 - 2(x+y) + (x-y)^2 \Big) \partial_y F(x,y,k) = (1 - 2k) (y - x - 1) F(x,y,k)$$ and has initial value at $y=0$ $$F(x,0,k) = \sum_{r=0}^{\infty} \binom{k+r-1}{r} \frac{\Gamma(2k+r-1)}{\Gamma(k+r)} x^r = \frac{\Gamma(2k-1)}{\Gamma(k)} \sum_{r=0}^{\infty} \binom{2k+r-1}{r} x^r = \frac{\Gamma(2k-1)}{\Gamma(k) (1 - x)^{2k-1}}$$ by the binomial theorem. We omit the details.
\end{proof}

Using $\lambda \mu' = -\mu \lambda'$ and therefore $Q(\lambda) \langle \nu, \mu \rangle = \lambda \mu' (\nu \lambda' - \nu' \lambda)$ we obtain, for the non-normalized Gegenbauer polynomials $g_N$, $$g_N^{k-1/2}\Big( - \langle \nu, \mu \rangle, -\frac{Q(\mu)}{4Q(\lambda)} \langle \nu, \lambda \rangle^2 \Big) = \Big( - \frac{\lambda \mu'}{Q(\lambda)} \Big)^N g_N^{k-1/2}(\nu \lambda' - \nu' \lambda, (\nu \lambda' + \nu' \lambda)^2/4)$$ and after passing to generating functions and applying the lemma with $x = -\nu \lambda'$ and $y = \nu' \lambda$, \begin{align*} &\quad \sum_{N=0}^{\infty} g_N^{k-1/2}\Big( - \langle \nu, \mu \rangle, -\frac{Q(\mu)}{4Q(\lambda)} \langle \nu, \lambda \rangle^2 \Big) t^N \\ &= \Big( 1 + \frac{\lambda \mu'}{Q(\lambda)} (\nu \lambda' - \nu' \lambda) t + \frac{\lambda^2 \mu'^2}{4Q(\lambda)^2} (\nu \lambda' + \nu' \lambda)^2 t^2 \Big)^{1/2 - k} \\ &= \frac{(k-1)!}{(2k-2)!} \sum_{N=0}^{\infty} \Big( \frac{\lambda \mu' t}{2 Q(\lambda)} \Big)^N \sum_{r=0}^N (-1)^r \binom{k+N-1}{r} \binom{k+N-1}{N-r} \frac{\Gamma(2k+N-1)}{\Gamma(k+N)} (\nu \lambda')^r (\nu' \lambda)^{N-r}, \end{align*} i.e. \begin{align*} &\quad G_N^{k-1/2}\Big( - \langle \nu, \mu \rangle, -\frac{Q(\mu)}{4Q(\lambda)} \langle \nu, \lambda \rangle^2 \Big) \\ &=\frac{N! (k-1)! (2k+N-2)! \Gamma(k-1/2)}{(2k-2)!(k+N-1)! \Gamma(k-1/2+\lceil N/2 \rceil)} \Big(\frac{\lambda \mu'}{2 Q(\lambda)} \Big)^N \sum_{r=0}^N (-1)^r \binom{k+N-1}{r} \binom{k+N-1}{N-r} (\nu \lambda')^r (\nu' \lambda)^{N-r}. \end{align*} Cohen's operator has the effect on Fourier expansions of sending $e^{2\pi i (\nu' \lambda \tau_1 + \nu \lambda' \tau_2)}$ to exactly this series $\sum_{r=0}^N (-1)^r \binom{k+N-1}{r} \binom{k+N-1}{N-r} (\nu \lambda')^r (\nu' \lambda)^{N-r}$ multiplied by $e^{2\pi i \langle \nu/\sqrt{d_{\mathcal{O}}}, \lambda \rangle \tau}$. Finally, using the simplification $$\frac{N! (k-1)! (2k+N-2)! \Gamma(k-1/2)}{(2k-2)!(k+N-1)! \Gamma(k-1/2+\lceil N/2 \rceil)} = \frac{2^N N!}{(k+\lfloor N/2 \rfloor) ... (k+N-1)}$$we obtain the claim.
\end{proof}

\subsection{Siegel modular forms I} Let $\Lambda$ be the Lorentzian lattice of symmetric $(2 \times 2)$ integral matrices with quadratic form given by the determinant and with bilinear form $$\langle A,B \rangle =a_{11} b_{22} + a_{22} b_{11} - 2a_{12} b_{12} = \mathrm{tr}(AB^{adj}), \; A = \begin{psmallmatrix} a_{11} & a_{12} \\ a_{12} & a_{22} \end{psmallmatrix}, \; B = \begin{psmallmatrix} b_{11} & b_{12} \\ b_{12} & b_{22} \end{psmallmatrix} \in \Lambda.$$ Then $\Lambda \oplus \mathrm{II}_{1,1}$ is isometric to the lattice of integral antisymmetric $(4 \times 4)$ matrices which are orthogonal to $J = \begin{psmallmatrix} 0 & 0 & -1 & 0 \\ 0 & 0 & 0 & -1 \\ 1 & 0 & 0 & 0 \\ 0 & 1 & 0 & 0 \end{psmallmatrix}$ with respect to the \emph{Pfaffian form} (a quadratic form which squares to the determinant) and on which $$\mathrm{PSp}_4(\mathbb{Z}) = \{M \in \mathrm{PSL}_4(\mathbb{Z}): \; M^T J M = J\}$$ acts by conjugation. This action by conjugation identifies $\mathrm{PSp}_4(\mathbb{Z})$ with the orthogonal modular group $\Gamma_{\Lambda}$. \\

A Siegel modular form (of degree two) of weight $k \in \mathbb{N}_0$ is a holomorphic function $F : \mathbb{H}_2 \rightarrow \mathbb{C}$ satisfying $F(M \cdot z) = \mathrm{det}(cz + d)^k F(z)$ for all $M = \begin{psmallmatrix} a & b \\ c & d \end{psmallmatrix} \in \mathrm{PSp}_4(\mathbb{Z})$, where $M \cdot z = (az+b)(cz+d)^{-1}$. The holomorphy extends to cusps by Koecher's principle. Any Siegel modular form $F$ can be written out as a Fourier series $$F(z) = \sum_T c(T) q^T, \; c(T) \in \mathbb{C}, \; q^T = e^{2\pi i \mathrm{tr}(Tz)}$$ where $T$ runs through symmetric matrices with integral diagonal and half-integral off-diagonal entries; i.e. the dual lattice $\Lambda'$. The correspondence between Siegel modular forms and orthogonal modular forms is such that $\mathbb{H}_2$ is exactly the upper half-space $\mathbb{H}_{\Lambda}$ and that if $F(z) = \sum_T c(T) q^T$ is a Siegel modular form then $$F(\mathbf{z}) = \sum_{T \in \Lambda'} c(T) \mathbf{q}^{T}, \; \mathbf{q}^T = e^{2\pi i \langle T^{adj}, \mathbf{z} \rangle} = e^{2\pi i \mathrm{tr}(T\mathbf{z})}$$ is an orthogonal modular form of the same weight.

For any order $\mathcal{O}$ in a real-quadratic number field $K$, Siegel modular forms can be pulled back to Hilbert modular forms as in the previous subsection, i.e. for the group $\Gamma_K = \mathrm{PSL}_2(\mathcal{O}^{\#} \oplus \mathcal{O})$. (A good reference for this is Section 5 of \cite{MZ}.) Here the distinction between the Hilbert modular groups is important: if $K$ is such that $\mathrm{PSL}_2(\mathcal{O}) \backslash (\mathbb{H} \times \mathbb{H})$ and $\Gamma_K \backslash (\mathbb{H} \times \mathbb{H})$ are not isomorphic, then there is no modular embedding of $\mathrm{PSL}_2(\mathcal{O}) \backslash (\mathbb{H} \times \mathbb{H})$ in $\mathrm{PSp}_4(\mathbb{Z}) \backslash \mathbb{H}_2$ at all (by \cite{H2}) and therefore no pullback. 

Let $ax^2 + bx + c \in \mathbb{Z}[x]$ be a polynomial with splitting field $K$. Fix a root $\mu/\lambda \in K$ where $\lambda,\mu$ are algebraic integers and consider the order $\mathcal{O} := \mathbb{Z}[\lambda,\mu] \subseteq \mathcal{O}_K$ with discriminant $d_{\mathcal{O}}$. Define $A := \begin{psmallmatrix} a & -b/2 \\ -b/2 & c \end{psmallmatrix} \in \Lambda'$ and consider the embedding $$\phi : \mathbb{H} \times \mathbb{H} \longrightarrow \mathbb{H}_2 \cap A^{\perp}, \; \; \phi(\tau_1,\tau_2) := \begin{psmallmatrix} \lambda^2 \tau_1 + (\lambda')^2 \tau_2 & \lambda \mu \tau_1 + \lambda' \mu' \tau_2 \\ \lambda \mu \tau_1 + \lambda' \mu' \tau_2 & \mu^2 + (\mu')^2 \tau_2 \end{psmallmatrix} = \Omega^T \begin{psmallmatrix} \tau_1 & 0 \\ 0 & \tau_2 \end{psmallmatrix} \Omega, \; \text{where} \; \Omega := \begin{psmallmatrix} \lambda & \mu \\ \lambda' & \mu' \end{psmallmatrix}.$$ For any $M = \begin{psmallmatrix} \alpha & \beta \\ \gamma & \delta \end{psmallmatrix} \in \mathrm{PSL}_2(\mathcal{O}^{\#} \oplus \mathcal{O})$, $$\phi(M \cdot (\tau_1,\tau_2)) = \Psi(M) \cdot \phi(\tau_1,\tau_2), \; \text{where} \; \Psi(M) = \begin{pmatrix} \Omega^T & 0 \\ 0 & \Omega^{-1} \end{pmatrix} \begin{psmallmatrix} \alpha & 0 & \beta & 0 \\ 0 & \alpha' & 0 & \beta' \\ \gamma & 0 & \delta & 0 \\ 0 & \gamma' & 0 & \delta' \end{psmallmatrix} \begin{pmatrix} \Omega^{-T} & 0 \\ 0 & \Omega \end{pmatrix} \in \mathrm{Sp}_4(\mathbb{Z}).$$


In particular, if $F$ is a Siegel modular form then $F \circ \phi(\tau_1,\tau_2)$ is a Hilbert modular form. This should be thought of as the $0^{\mathrm{th}}$ pullback to the Humbert surface $A^{\perp}$.  We will write out the higher pullbacks. Any index $T = \begin{psmallmatrix} t_1 & t_2/2 \\ t_2/2 & t_3 \end{psmallmatrix} \in \Lambda'$ has orthogonal projection to $A^{\perp}$ given by $$\frac{\langle T,A \rangle}{\langle A,A \rangle} A = -\frac{a t_1 - b t_2/2 + c t_3}{d_{\mathcal{O}}}A$$ with norm $-(at_1 - bt_2/2 + ct_3)^2 / d_{\mathcal{O}}$. Therefore, if $F(z) = \sum_T c(T) \mathbf{q}^T$ is a Siegel modular form then its $N^{\text{th}}$ pullback to $A^{\perp}$ is \begin{align*} P_N^{A^{\perp}} F(\tau_1,\tau_2) &= \sum_T c(T) G_N^{k-1}\Big( -2 (at_1 - bt_2/2 + ct_3), d_{\mathcal{O}} \mathrm{det}(T) + (at_1 - bt_2/2 + ct_3)^2 \Big) \mathbf{q}^{\lambda^2 t_1 + \lambda\mu t_2 + \mu^2 t_3}, \end{align*} where $\mathbf{q}^{\nu} = e^{2\pi i (\nu \tau_1 + \nu' \tau_2)}$ as usual.

\begin{rem} Similar pullbacks can be written down for divisors $A^{\perp}$ where $\mathrm{det}(A)$ is a negative square (which, roughly speaking, corresponds to taking $\mathcal{O} \subseteq K := \mathbb{Q} \oplus \mathbb{Q}$). A particularly simple case is $A = \begin{psmallmatrix} 0 & 1 \\ 1 & 0 \end{psmallmatrix}$, where $\mathbb{H}_2 \cap A^{\perp}$ is the diagonal; the pullbacks in this case are tensor products of elliptic modular forms of level one, given by the formula $$P_N^{A^{\perp}} F(\tau_1,\tau_2) = \sum_T c(T) G_N^{k-1}(t_2,t_1 t_3) q_1^{t_1} q_2^{t_3}.$$ This is a special case of the more general operator defined by Ibukiyama (compare \cite{I}, 3.1.1).
\end{rem}

\begin{ex} We will work out an example involving the pullbacks of a meromorphic modular form to a Heegner divisor along which it has a pole following Remark 3.10. Let $q_1 = e^{2\pi i \tau_1}$, $r = e^{2\pi i w}$, $q_2 = e^{2\pi i \tau_2}$ and let $$\Psi_{10}\left( \begin{psmallmatrix} \tau_1 & w \\ w & \tau_2 \end{psmallmatrix}\right) = q_1 q_2 (r^{1/2} - r^{-1/2})^2 \Big( 1 - 2 (r + 10 + r^{-1})(q_1 + q_2) + O(q_1,q_2)^2 \Big)$$ be Igusa's weight 10 cusp form on $\mathbb{H}_2$: that is, up to a multiple, the discriminant of the genus two curve whose Jacobian has modulus $\tau = \begin{psmallmatrix} \tau_1 & w \\ w &\tau_2 \end{psmallmatrix}$ if $\tau$ is not equivalent to any point with $w = 0$, and it has a double zero along the diagonal. The physical meaning of the Fourier-Jacobi coefficients of $\Psi_{10}^{-1}$ is well-known (e.g. \cite{DMZ}). It was recently proved \cite{OP} that the \emph{Taylor} coefficients of $\Psi_{10}^{-1}$ about the divisor $w=0$ have a beautiful interpretation in terms of the reduced Gromov-Witten theory of spaces $X = S \times E$ where $S$ is a projective K3 surface and $E$ is an elliptic curve:

$$F\left( \begin{psmallmatrix} \tau_1 & w \\ w & \tau_2 \end{psmallmatrix} \right) := -w^2 \Psi_{10}^{-1}\left( \begin{psmallmatrix} \tau_1 & w \\ w & \tau_2 \end{psmallmatrix} \right) = \sum_{g=0}^{\infty} \sum_{h=0}^{\infty} \sum_{d=0}^{\infty} (2\pi)^{2g} N_{g,h,d}  q_1^{h-1} q_2^{d-1} w^{2g},$$ where $N_{g,h,d}$ is the Gromov-Witten invariant counting genus $g$ curves on $X$ up to translation in a fixed class $(\beta,d) \in H_2(X,\mathbb{Z})$. Here $\beta \in H_2(S,\mathbb{Z})$ is a primitive lattice vector satisfying $\beta \cdot \beta = 2h-2$, and $d \in \mathbb{N}_0$ is understood as an element of $H_2(E;\mathbb{Z})$.

We obtain meromorphic modular forms for $\mathrm{SL}_2(\mathbb{Z}) \times \mathrm{SL}_2(\mathbb{Z})$ (in fact, holomorphic modular forms divided by $\Delta(\tau_1)\Delta(\tau_2)$) by pretending that $F$ is modular of weight $(-12)$ and formally applying Ibukiyama's operators: $$P_N^{w=0} F(\tau_1,\tau_2) = (2\pi i)^{-N} G_N^{-13}(\partial_w, \partial_{\tau_1} \partial_{\tau_2}) \Big|_{w=0} F\left( \begin{psmallmatrix} \tau_1 & w \\ w & \tau_2 \end{psmallmatrix} \right).$$ Explicitly these pullbacks have Fourier expansions $$P_N^{w=0} F(\tau_1,\tau_2) = \sum_{h=0}^{\infty} \sum_{d=0}^{\infty} \sum_{\substack{2(n_1 + g) = N \\ n_1 + 2g \le 13 \\ n_1,g \ge 0}} (-1)^g \frac{13!}{n_1! (2g)! (13 - n_1 - 2g)!} N_{g,h,d}  (h-1)^{n_1} (d-1)^{n_1} q_1^{h-1} q_2^{d-1}.$$ These vanish trivially when $N > 26$ or when $N$ is odd. In the first two nontrivial cases we find: \\
(i) When $N = 0$, $$P_0^{w=0} F(\tau_1,\tau_2) = \sum_{h,d=0}^{\infty} N_{0,h,d} q_1^{h-1} q_2^{d-1} = \frac{1}{\Delta(\tau_1)\Delta(\tau_2)};$$
(ii) When $N=2$, the nonexistence of (holomorphic) modular forms for $\mathrm{SL}_2(\mathbb{Z}) \times \mathrm{SL}_2(\mathbb{Z})$ of weight two yields $$P_2^{w=0} F(\tau_1,\tau_2) = 13 \sum_{h,d=0}^{\infty} \Big(  -6 N_{1,h,d} + (h-1)(d-1) N_{0,h,d} \Big) q_1^{h-1} q_2^{d-1} = 0.$$ Since $\Psi_{10}$ is an additive theta lift of a Jacobi cusp form $\phi_{10, 1}$ of weight 10 (cf. \cite{GN2}), the pulled-back forms can also be expressed as theta lifts of development coefficients of $\phi_{10, 1}$.
\end{ex}

\subsection{Siegel modular forms II} Let $A$ be a symmetric positive-definite $(2 \times 2)$ matrix with integral entries. If $F : \mathbb{H}_2 \rightarrow \mathbb{C}$ is a Siegel modular form of weight $k$ then $f(\tau) := F(A\tau)$, $\tau \in \mathbb{H}$ satisfies $$f\Big(\frac{a \tau + b}{c \tau + d} \Big) = F \Big( \begin{psmallmatrix} aI & bA \\ cA^{-1} & dI \end{psmallmatrix} \cdot (A\tau) \Big) = (c\tau + d)^{2k} f(\tau) \; \text{for all} \; \begin{psmallmatrix} a & b \\ c & d \end{psmallmatrix} \in \Gamma_0(\mathrm{det}\, A).$$ This should be thought of as the $0^{\mathrm{th}}$ pullback from orthogonal modular forms for $\mathrm{SO}^+(2,3)$ to an embedded curve corresponding to $\mathrm{SO}^+(2,1)$. We will consider the higher pullbacks. \\

The orthogonal projection of $T^{adj} \in \Lambda'$ to $A$ is $\frac{\mathrm{tr}(TA)}{2\mathrm{det}(A)}A$ with norm $\frac{\mathrm{tr}(TA)^2}{4 \mathrm{det}(A)}$. For any $B \in A^{\perp}$ we obtain the $N^{\text{th}}$ pullback in the direction $B \otimes ... \otimes B$ of $F(z) = \sum_T c(T) \mathbf{q}^T$ as $$P_N F(\tau;B \otimes ... \otimes B) = \sum_T c(T) G_N^{k-1/2}\Big( \mathrm{tr}(TB), -\frac{\mathrm{det}(B)\mathrm{tr}(TA)^2}{4 \mathrm{det}(A)} \Big) q^{\mathrm{tr}(TA)}, \; \; q^n = e^{2\pi i \tau},$$ and we obtain the pullbacks in general directions by polarization. \\

For example, the lowest order pullbacks to $A\mathbb{H} \subseteq \mathbb{H}_2$ are $$P_0 F(\tau) = F(A\tau) = \sum_T c(T) q^{\mathrm{tr}(TA)} \in M_{2k}(\Gamma_0(\mathrm{det}(A));$$ $$P_1 F(\tau;B) =  \sum_T c(T) \mathrm{tr}(TB) q^{\mathrm{tr}(TA)} \in S_{2k+2}(\Gamma_0(\mathrm{det}(A)));$$ \begin{align*} P_2 F(\tau;B_1 \otimes B_2)  &= \frac{1}{4\mathrm{det}(A)} \sum_T c(T) \Big( (4k+2) \mathrm{det}(A) \mathrm{tr}(TB_1) \mathrm{tr}(TB_2) + \mathrm{tr}(B_1^{adj}B_2) \mathrm{tr}(TA)^2) \Big) q^{\mathrm{tr}(TA)} \\ &\in S_{2k+4}(\Gamma_0(\mathrm{det}(A))); \end{align*} and \begin{align*} P_3 F(\tau;B_1 \otimes B_2 \otimes B_3) &= \frac{1}{4 \mathrm{det}(A)} \sum_T c(T) \Big[ (4k+6) \mathrm{det}(A) \mathrm{tr}(TB_1) \mathrm{tr}(TB_2) \mathrm{tr}(TB_3) \\ &\quad + \Big(\mathrm{tr}(B_1^{adj} B_2) \mathrm{tr}(TB_3) + \mathrm{tr}(B_2^{adj} B_3) \mathrm{tr}(TB_1) + \mathrm{tr}(B_3^{adj} B_1) \mathrm{tr}(TB_2)\Big) \mathrm{tr}(TA)^2 \Big] q^{\mathrm{tr}(TA)} \\ & \in S_{2k+6}(\Gamma_0(\mathrm{det}(A))). \end{align*} Here $B$, $B_i$ are any symmetric (but not positive definite) complex matrices with $\mathrm{tr}(B_i^{adj} A) = 0$. Note that if $\mathrm{det}(A) = 1$ then all odd-order pullbacks are zero.

\begin{ex} (See also \cite{GN}, Example 3.9) As a numerical example, let $q = e^{2\pi i \tau}, \; r = e^{2\pi i z}, \; s = e^{2\pi i w}$ and let \begin{align*} \Psi_{35}\left( \begin{psmallmatrix} \tau & z \\ z & w \end{psmallmatrix} \right) &= q^2 s^2 (q-s) (r - r^{-1}) \Big[ 1 - (q+s) (r^2 + 70 + r^{-2}) + 69(q^2 + s^2)(r^2 + 33 + r^{-2}) \\ &\quad + qs(r^4 + 70r^2 - 32384r - 127074 - 32384r^{-1} + 70r^{-2} + r^{-4}) + O(q,s)^3 \Big] \end{align*} be the unique (up to scalar) Siegel modular form of weight $35$. Take $A = \begin{psmallmatrix} 1 & 0 \\ 0 & 1 \end{psmallmatrix}$ and fix the basis $B_1 = \begin{psmallmatrix} 1 & 0 \\ 0 & -1 \end{psmallmatrix}$ and $B_2 = \begin{psmallmatrix} 0 & 1 \\ 1 & 0 \end{psmallmatrix}$ of $A^{\perp} \otimes \mathbb{C}$. Then $P_0 \Psi_{35}$, $P_1 \Psi_{35},$ $P_3 \Psi_{35}$ and $P_2 \Psi_{35}(\tau;B_1 \otimes B_1), P_2 \Psi_{35}(\tau;B_2 \otimes B_2)$ are zero; but $$P_2 \Psi_{35}(\tau;B_1 \otimes B_2) = P_2 \Psi_{35}(\tau;B_2 \otimes B_1) = 71q^5 - 10224q^6 - 13257972q^7 \pm ... = 71 E_4^2 E_6 \Delta^5 \in S_{74}(\mathrm{SL}_2(\mathbb{Z})).$$
\end{ex}

\bibliographystyle{plainnat}
\bibliofont
\bibliography{\jobname}

\end{document}